\documentclass{article}
\usepackage{graphicx} 

\usepackage{amsmath,amssymb,amsfonts}%
\usepackage{amsthm}%
\usepackage{mathrsfs}%
\usepackage{booktabs}
\usepackage{amsfonts}
\usepackage{graphicx}
\usepackage{epstopdf}
\usepackage{algorithmic}
\usepackage{mathdots}
\usepackage{multirow}
\usepackage[shortlabels]{enumitem}
\usepackage{comment}
\usepackage{rotating}
\usepackage{authblk}

\usepackage[colorlinks,linkcolor=blue,anchorcolor=blue,citecolor=blue]{hyperref}
\usepackage{cleveref}

\theoremstyle{definition}
\newtheorem{definition}{Definition}[section]
\newtheorem{theorem}{Theorem}[section]
\newtheorem{lemma}[theorem]{Lemma}
\newtheorem{proposition}[theorem]{Proposition}
\newtheorem{corollary}[theorem]{Corollary}

\newtheorem{remark}{Remark}
\newtheorem{example}{Example}

\usepackage{caption}
\usepackage{subcaption}

\crefname{hypothesis}{Hypothesis}{Hypotheses}

\newcommand{\mat}[1]{\begin{bmatrix}#1 \\ \end{bmatrix}}

\title{Block $\omega$-circulant preconditioners for parabolic optimal control problems}

\author{Po Yin Fung\thanks{Department of Mathematics, Hong Kong Baptist University, Kowloon Tong, Hong Kong SAR (\href{pyfung@hkbu.edu.hk}{pyfung@hkbu.edu.hk}).}~~and Sean Hon\thanks{Corresponding Author. Department of Mathematics, Hong Kong Baptist University, Hong Kong SAR (\href{seanyshon@hkbu.edu.hk}{seanyshon@hkbu.edu.hk}).} }

\begin{document}

\date{}

\maketitle

\begin{abstract}
In this work, we propose a class of novel preconditioned Krylov subspace methods for solving an optimal control problem of parabolic equations. Namely, we develop a family of block $\omega$-circulant based preconditioners for the all-at-once linear system arising from the concerned optimal control problem, where both first order and second order time discretization methods are considered. The proposed preconditioners can be efficiently diagonalized by fast Fourier transforms in a parallel-in-time fashion, and their effectiveness is theoretically shown in the sense that the eigenvalues of the preconditioned matrix are clustered around $\pm 1$, which leads to rapid convergence when the minimal residual method is used. When the generalized minimal residual method is deployed, the efficacy of the proposed preconditioners are justified in the way that the singular values of the preconditioned matrices are proven clustered around unity. Numerical results are provided to demonstrate the effectiveness of our proposed solvers.
\end{abstract}

\noindent{\bf Keywords:} Toeplitz, skew/circulant matrices, $\omega$-circulant matrices, preconditioners, parallel-in-time

\bigskip

\noindent{\bf AMS Subject Classification:} 65F08, 65F10, 65M22, 15B05



\section{Introduction}\label{sec:Introduction}

In recent years, there has been a growing interest in analyzing and solving optimization problems constrained by parabolic equations. We refer to \cite{lions_1971, Hinze_2008, Troltzsch2010, Borzi_2011} and the references therein for a comprehensive overview.

In this work, we are interested in solving the distributed optimal control model problem. Namely, the following quadratic cost functional is minimized:
\begin{equation}\label{eqn:Cost_functional_heat}
\min_{y,u}~ \mathcal{J}(y,u):=\frac{1}{2}\| y - g \|^{2}_{L^2(\Omega \times (0,T))} + \frac{\gamma}{2}\| u \|^{2}_{L^2(\Omega \times (0,T))},
\end{equation}
subject to a parabolic equation with certain initial and boundary conditions
\begin{equation}\label{eqn:heat}
\left\{
\begin{array}{lc}
 y_{t} - \mathcal{L} y = f + u, \quad (x,t)\in \Omega \times (0,T], \qquad  y = 0, \quad (x,t)\in \partial \Omega \times (0,T], \\
y(x,0)=y_0, \quad x \in \Omega,
\end{array}
	\right.\,
\end{equation}
where $u,g \in L^2$ are the distributed control and the desired tracking trajectory, respectively, $\gamma>0$ is a regularization parameter, $\mathcal{L}=\nabla \cdot (a({x})\nabla )$, and $f$ and $y_0$ are given problem dependent functions. Under appropriate assumptions, theoretical aspects such as the existence, uniqueness, and regularity of the solution were well studied in \cite{lions_1971}. The optimal solution of (\ref{eqn:Cost_functional_heat}) \& (\ref{eqn:heat}) can be characterized by the following system:
\begin{equation}\label{eqn:heat_2}
\left\{
\begin{array}{lc}
 y_{t} - \mathcal{L} y - \frac{1}{\gamma} p= f,\quad (x,t)\in \Omega \times (0,T], \qquad y = 0,\quad  (x,t)\in \partial \Omega \times (0,T], \\
y(x,0)=y_0 \quad  x \in \Omega,\\
-p_{t} - \mathcal{L} p + y = g, \quad (x,t)\in \Omega \times (0,T], \qquad p = 0, \quad (x,t) \in \partial \Omega \times (0,T], \\
p(x,T) = 0 \quad  x \in \Omega,
\end{array}
	\right.\,
\end{equation}
where the control variable $u$ has been eliminated. 

Following \cite{WuZhou2020, WuWangZhou2023}, we discretize (\ref{eqn:heat_2}) using the $\theta$-method for time and some space discretization, which gives
\begin{eqnarray*}
M_{m}\frac{\mathbf{y}_m^{(k+1)} - \mathbf{y}_m^{(k)} }{\tau} +  K_m  (\theta\mathbf{y}_m^{(k+1)} + (1-\theta)\mathbf{y}_m^{(k)})  &=&  M_{m}(\theta\mathbf{f}_m^{(k+1)} + (1-\theta)\mathbf{f}_m^{(k)} + \frac{1}{\gamma} (\theta\mathbf{p}_m^{(k)} + (1-\theta)\mathbf{p}_m^{(k+1)}) ), \\
-M_{m}\frac{\mathbf{p}_m^{(k+1)} - \mathbf{p}_m^{(k)} }{\tau} +  K_m (\theta\mathbf{p}_m^{(k)} + (1-\theta)\mathbf{p}_m^{(k+1)})    &=&  M_{m}( \theta\mathbf{g}_m^{(k)} + (1-\theta)\mathbf{g}_m^{(k+1)} - \theta\mathbf{y}_m^{(k+1)} - (1-\theta)\mathbf{y}_m^{(k)} ).
\end{eqnarray*} The backward Euler method corresponds to $\theta=1$, while the Crank-Nicolson method is adopted when $\theta=1/2$.

Combining the given initial and boundary conditions, one needs to solve the following linear system 
\begin{equation}\label{eqn:main_system_before}
\widetilde{\mathcal{A}} \begin{bmatrix} \mathbf{y}\\ \mathbf{p} \end{bmatrix} = \begin{bmatrix} \mathbf{\widetilde{g}}\\  \mathbf{\widetilde{f}} \end{bmatrix},
\end{equation}
where we have $\mathbf{y} = [ \mathbf{y}_m^{(1)}, \cdots, \mathbf{y}_m^{(n)}]^{\top}$, $\mathbf{p} = [ \mathbf{p}_m^{(0)},\cdots, \mathbf{p}_m^{(n-1)}]^{\top}$, 
\begin{equation*}
    \mathbf{\widetilde{f}}=\mat{M_m (\theta\tau\mathbf{f}_m^{(1)} + (1-\theta)\tau\mathbf{f}_m^{(0)})+(M_{m} - (1-\theta)\tau K_m)\mathbf{y}_m^{(0)}\\
    M_m (\theta\tau\mathbf{f}_m^{(2)} + (1-\theta)\tau\mathbf{f}_m^{(1)})\\
    \vdots\\
    M_m (\theta\tau\mathbf{f}_m^{(n)} + (1-\theta)\tau\mathbf{f}_m^{(n-1)})},
    \mathbf{\widetilde{g}} = \tau\mat{M_m (\theta\mathbf{g}_m^{(0)} + (1-\theta)\mathbf{g}_m^{(1)} - (1-\theta)\mathbf{y}_m^{(0)})\\
    M_m (\theta\mathbf{g}_m^{(1)} + (1-\theta)\mathbf{g}_m^{(2)})\\
    \vdots\\
    M_m (\theta\mathbf{g}_m^{(n-1)} + (1-\theta)\mathbf{g}_m^{(n)})},
\end{equation*}
\begin{eqnarray}\label{eqn:matrix_A_before}
\widetilde{\mathcal{A}} &=&
 \begin{bmatrix} 
\tau B_n^{(2)} \otimes M_{m}  &   (B_n^{(1)})^{\top} \otimes M_{m} + \tau  (B_n^{(2)})^{\top} \otimes K_m \\
 B_n^{(1)} \otimes M_{m} + \tau  B_n^{(2)} \otimes K_m  &  -\frac{\tau}{\gamma} (B_n^{(2)})^{\top} \otimes M_{m}
\end{bmatrix},
\end{eqnarray}
and the matrices $B_n^{(1)}, B_n^{(2)}$ are, respectively,
\begin{equation*}
B_n^{(1)} = \begin{bmatrix}
1 &   &  &  & \\
-1  & 1    & & & \\
   &  -1  & 1  & &  \\
  &     & \ddots & \ddots &  \\
 &  &    & -1 & 1
\end{bmatrix}, \quad
B_n^{(2)} = \mat{
\theta& & & & \\
1-\theta&\theta& & & \\
 &1-\theta&\theta& & \\
 & &\ddots&\ddots& \\
 & & &1-\theta&\theta}.
\end{equation*}

We assume that the matrix $M_{m}$ is symmetric positive definite, and the matrix $K_{m}$ is symmetric positive semi-definite. The matrices $M_{m}$ and $K_{m}$ represent the mass matrix and the stiffness matrix, respectively, if a finite element method is employed. For the finite difference method, the linear system is such that $M_{m}=I_{m}$ and $K_{m}=-L_{m}$, where $-L_{m}$ is the discretization matrix of the negative Laplacian.

Following a similar idea in \cite{LevequePearson22}, we can further transform (\ref{eqn:matrix_A_before}) into the following equivalent system
\begin{equation}\label{eqn:main_system}
{\mathcal{A}}\begin{bmatrix} \sqrt{\gamma}\widetilde{\mathbf{y}}\\ \widetilde{\mathbf{p}} \end{bmatrix} = \begin{bmatrix} \mathbf{g}\\ \sqrt{\gamma} \mathbf{f} \end{bmatrix},
\end{equation}
where $\mathbf{f} =(I_{n} \otimes M_{m}^{-\frac{1}{2}}) \mathbf{\widetilde{f}}$, $\mathbf{g} =(I_{n} \otimes M_{m}^{-\frac{1}{2}}) \mathbf{\widetilde{g}}$, $\widetilde{\mathbf{y}} = (B_n^{(2)}\otimes M_{m}^{\frac{1}{2}}) [ \mathbf{y}_m^{(1)}, \cdots, \mathbf{y}_m^{(n)}]^{\top}$, $\mathbf{p} = ((B_n^{(2)})^{\top}\otimes M_{m}^{\frac{1}{2}}) [ \mathbf{p}_m^{(0)},\cdots, \mathbf{p}_m^{(n-1)}]^{\top}$, and
\begin{eqnarray}\label{eqn:matrix_A}
{\mathcal{A}} &=&
 \begin{bmatrix} 
\alpha {I}_n \otimes I_m  & B_n^{\top}\otimes I_m + \tau  {I}_n \otimes M_{m}^{-\frac{1}{2}} K_{m} M_{m}^{-\frac{1}{2}}\\
 B_n\otimes I_m + \tau  {I}_n\otimes M_{m}^{-\frac{1}{2}} K_{m} M_{m}^{-\frac{1}{2}}  &  -\alpha {I}_n \otimes I_m 
\end{bmatrix}\\\nonumber
&=&
\begin{bmatrix} 
\alpha {I}_n \otimes I_m  & \mathcal{T}^{\top}\\
 \mathcal{T}  & -\alpha {I}_n \otimes I_m 
\end{bmatrix}.
\end{eqnarray}

Note that $\frac{1}{2} \leq \theta \leq 1$, {$\alpha = \frac{\tau}{\sqrt{\gamma}}$}, $I_m$ is the $m \times m$ identity matrix, and
\begin{equation}\label{eqn:matrix_T}
\mathcal{T}=B_n\otimes I_m + \tau  {I}_n\otimes M_{m}^{-\frac{1}{2}} K_{m} M_{m}^{-\frac{1}{2}},
\end{equation}
where $B_n$ is a lower triangular Toeplitz matrix whose entries are known explicitly as
\begin{equation*} 
B_n = \begin{bmatrix}
\frac{1}{\theta} &   &  &  & \\
\frac{-1}{\theta^2}  & \frac{1}{\theta}   & & & \\
\frac{-(\theta-1)}{\theta^3}  & \frac{-1}{\theta^2}   & \frac{1}{\theta}  & &  \\
\vdots  & \ddots    & \ddots & \ddots &  \\
\frac{-(\theta-1)^{n-2}}{\theta^n} & \cdots & \frac{-(\theta-1)}{\theta^3}    & \frac{-1}{\theta^2}  & \frac{1}{\theta} 
\end{bmatrix}.
\end{equation*}
Incidentally, $B_n$ can be expressed as the product of two Toeplitz matrices, i.e., $B_n=B_n^{(1)}(B_n^{(2)})^{-1}=(B_n^{(2)})^{-1}B_n^{(1)}$.
As will be explained in Section \ref{sec:prelim}, the Toeplitz matrices $B_n^{(1)}$ and $B_n^{(2)}$ are respectively generated by the functions 
\begin{equation}\label{eqn:functionB1}
 f_1( \phi ) = 1 - \exp{(\mathbf{i}\phi})
\end{equation}
and
\begin{equation}\label{eqn:functionB2}
f_2( \phi ) = \theta + (1-\theta) \exp{(\mathbf{i}\phi)}.
\end{equation}

In what follows, we focus on using the finite difference method to discretize the system (\ref{eqn:heat_2}), namely, $M_{m}=I_{m}$ and $K_{m}=-L_{m}$ in the linear system (\ref{eqn:main_system}). However, we point out that our proposed preconditioning methods with minimal modification are still applicable when a finite element method is used. We first develop a preconditioned generalized minimal residual (GMRES) method for a nonsymmetric equivalent system of (\ref{eqn:main_system}), which is 
\begin{equation}\label{eqn:main_system_nonsym}
\mathcal{\widehat{A}}\begin{bmatrix}  \sqrt{\gamma} \widetilde{\mathbf{y}} \\ \widetilde{\mathbf{p}} \end{bmatrix} = \begin{bmatrix}  \sqrt{\gamma} \mathbf{f} \\ \mathbf{g} \end{bmatrix}
\end{equation}
where 
\begin{eqnarray}\label{eqn:matrix_A_hat}
\mathcal{\widehat{A}}&=&
\begin{bmatrix} 
 \mathcal{T}  & -\alpha {I}_n \otimes I_m \\
\alpha {I}_n \otimes I_m  & \mathcal{T}^{\top}
\end{bmatrix}.
\end{eqnarray} For $\mathcal{\widehat{A}}$, we propose the following novel block preconditioner:
\begin{eqnarray}\label{eqn:matrix_P_GMRES}
\mathcal{P}_{S} = 
\begin{bmatrix} 
 \mathcal{S}  & -\alpha {I}_n \otimes I_m \\
\alpha {I}_n \otimes I_m  & \mathcal{S}^* 
\end{bmatrix},
\end{eqnarray}
where 
\begin{equation}\label{eqn:matrix_S}
\mathcal{S}=S_n\otimes I_m + \tau  I_n\otimes (-L_m).
\end{equation}
Notice that $S_n:=S_n^{(1)}(S_n^{(2)})^{-1}$, where
\begin{eqnarray*}\label{eqn:matrix_strang_1}
S_n^{(1)} = \begin{bmatrix}
1 &   &  &   & -\omega\\
-1  & 1    & & &  \\
   &  -1  & 1  & &  \\
  &     & \ddots & \ddots &  \\
 &  &    & -1 & 1
 \end{bmatrix}, \quad
S_n^{(2)} = \begin{bmatrix}
\theta &   &  &   & \omega(1-\theta)\\
1-\theta  & \theta    & & &  \\
   &  1-\theta  & \theta  & &  \\
  &     & \ddots & \ddots &  \\
 &  &    & 1-\theta & \theta
 \end{bmatrix}
\end{eqnarray*}
and $\omega = e^{\textbf{i}\zeta} \in \mathbb{C}$ with $\zeta \in [0,2\pi)$. Clearly, both $S_n^{(1)}$ and $S_n^{(2)}$ are $\omega$-circulant matrices \cite{BiniLatoucheMeini,Bertaccini_Ng_2003}, so they admit the eigendecompositions 
\begin{equation}\label{eqn:decompositiob_S12}
	S_n^{(k)} = (\Gamma_n\mathbb{F}_{n}) \Lambda_{n}^{(j)} (\Gamma_n\mathbb{F}_{n})^{*} , \quad j=1,2,
\end{equation}
where $\Gamma_n={\rm diag} ( \exp{(  \textbf{i}\zeta{(\frac{k-1}{n}})   )}   )_{k=1}^{n}$ and $\mathbb{F}_n=\frac{1}{\sqrt{n}}[\theta_n^{(i-1)(j-1)}]_{i,j=1}^{n}$ with $\theta_n =\exp(\frac{2\pi {\bf i}}{n})$ and
	$\Lambda_{n}^{(j)}={\rm diag} (f_1 ( \frac{\zeta + 2\pi k}{n}  ) )_{k=0}^{n-1}$ \textrm{and} $f_j$ is defined by \eqref{eqn:functionB1} and \eqref{eqn:functionB2}.

\begin{remark}
Since $S_n^{(2)}$ is an $\omega$-circulant matrix, its eigenvalues $\lambda_k(S_n^{(2)})$ can be found explicitly, i.e., $\lambda_k(S_n^{(2)}) = \theta + (1-\theta)\exp{\big(\mathbf{i}(\frac{\zeta + 2\pi k}{n})\big)}, k = 0,\dots, n-1 $. It is known that $S_n^{(2)}$ can be singular. For example, $S_n^{(2)}$ has a zero eigenvalue for even $n$ when $\theta=\frac{1}{2}$ and $\zeta=0$ (i.e., $\omega=1$), which was discussed in \cite{WuZhou2020}. When it happens, a remedy is to replace the zero eigenvalue by a nonzero real number. Thus, it can easily create a nonsingular circulant matrix $\widetilde{S}_n^{(2)}$ such that $\mathrm{rank}(\widetilde{S}_n^{(2)}-{S}_n^{(2)})=1$. Hence, our preconditioning approach can work with $S_n$ replaced by $\widetilde{S}_n=S_n^{(1)}(\widetilde{S}_n^{(2)})^{-1}$, without the restrictive assumption needed in \cite{WuZhou2020} (i.e., $n$ should be chosen odd). Therefore, for ease of exposition, we assume that $S_n^{(2)}$ is nonsingular in the rest of this work.
\end{remark}

\begin{remark}
    It should be noted that our adopted $\omega$-circulant preconditioning is distinct from the $\epsilon$-circulant preconditioning that has received much attention in the literature due to its excellent performance for solving PDE problems (see, e.g, \cite{doi:10.1137/20M1316354, doi:10.1137/19M1309869, Sun2022, 2020arXiv200509158G}), despite these two kind of matrices have a similar decomposition like \eqref{eqn:decompositiob_S12}. A notable difference between the two is that $\omega$ is a complex number in general, while $\epsilon$ is only chosen real. Correspondingly, the diagonal matrix $\Gamma_n$ for $\omega$-circulant matrices is unitary in general, while that for $\epsilon$-circulant matrices is not.
\end{remark}

 When $\omega=1$, $\mathcal{P}_{S}$ becomes the block circulant based preconditioner proposed in \cite{WuZhou2020}. The existing block skew-circulant based preconditioner \cite{Bouillon2021} proposed only with the backward Euler method is also included in our preconditioning strategy when $\omega=-1$ and $\theta=1$. As extensively studied in \cite{PottsSteidl1991,PottsSteidl2001}, $\omega$-circulant matrices as preconditioners for Toeplitz systems can substantially outperform the Strang type preconditioners \cite{Strang1986} (i.e., when $\omega=1$), especially in the ill-conditioned case. For related studies on the unsatisfactory performance of Strang preconditioners for ill-conditioned nonsymmetric (block) Toeplitz systems, we refer to \cite{Hon_SC_Wathen,HonFungDongSC_2023}.

 Since $\omega$-circulant matrices can be efficiently diagonalized by the fast Fourier transforms (FFTs), which can be parallelizable over different possessors. Hence, our preconditioner $\mathcal{P}_S$ is especially advantageous in a high performance computing environment.

In order to support our GMRES solver with $\mathcal{P}_S$ as a preconditioner, we will show that the singular values of $\mathcal{P}_S^{-1} \mathcal{\widehat{A}}$ are clustered around unity. However, despite of its success which can be seen in the numerical experiments from Section \ref{sec:numerical}, the convergence study of preconditioning strategies for nonsymmetric problems is to a great extent heuristic. As mentioned in \cite[Chapter 6]{ANU:9672992}, descriptive convergence bounds are usually not available for GMRES or any of the other applicable nonsymmetric Krylov subspace iterative methods. 

Therefore, as an alternative solver, we develop a preconditioned minimal residual (MINRES) method for the symmetric system (\ref{eqn:main_system}), instead of \eqref{eqn:main_system_nonsym}. Notice that our proposed MINRES method is in contrast with the aforementioned GMRES solvers, such as \cite{WuZhou2020,Bouillon2021} where a block (skew-)circulant type preconditioner was proposed and the eigenvalues of the preconditioned matrix were shown clustered around unity. As well explained in \cite{Greenbaum_1996}, the convergence behaviour of GMRES cannot be rigorously analyzed by using only eigenvalues in general. Thus, our MINRES solver can get round these theoretical difficulties of GMRES. 

Based on the spectral distribution of $\mathcal{A}$, we first propose the following novel SPD block diagonal preconditioner as an ideal preconditioner for $\mathcal{A}$:
\begin{eqnarray}\label{eqn:abs_ideal_matrix_H}
|\mathcal{A}|:=\sqrt{\mathcal{A}^2}=\begin{bmatrix} 
 \sqrt{ \mathcal{T}^{\top}\mathcal{T} + \alpha^2 {I}_n \otimes I_m } & \\
   & \sqrt{ \mathcal{T} \mathcal{T}^{\top}+ \alpha^2 {I}_n \otimes I_m}
\end{bmatrix}.
\end{eqnarray}

Despite its excellent preconditioning effect for $\mathcal{A}$, which will be shown in Section \ref{sub:ideal_preconditioner}, the matrix $|\mathcal{A}|$ is computational expensive to invert. Thus, we then propose the following parallel-in-time (PinT) preconditioner, which mimics $|\mathcal{A}|$ and can be fast implemented:
\begin{eqnarray}\label{eqn:matrix_P}
|\mathcal{P}_{S}| :=\sqrt{\mathcal{P}_{S}^* \mathcal{P}_{S}} = \begin{bmatrix} 
\sqrt{\mathcal{S}^* \mathcal{S} +  \alpha^2 {I}_n \otimes I_m} 
&  \\
  & \sqrt{\mathcal{S} \mathcal{S}^* +  \alpha^2 {I}_n \otimes I_m} 
\end{bmatrix}.
\end{eqnarray}

However, the preconditioner $|\mathcal{P}_S|$ requires fast diagonalizability of $L_m$ in order to be efficiently implemented. When such diagonalizability is not available, we further propose the following preconditioner $\mathcal{P}_{MS}$ as a modification of $|\mathcal{P}_S|$:
\begin{eqnarray}\label{eqn:matrix_P_MS}
&&\mathcal{P}_{MS} \\ \nonumber
&=& \begin{bmatrix} 
\sqrt{S_{n}^* S_n +  \alpha^2 I_{n}} \otimes I_{m} + \tau  I_n\otimes (-L_m)
&  \\
  & \sqrt{S_n S_{n}^* +  \alpha^2 I_{n}} \otimes I_{m} + \tau I_{n} \otimes (-L_m)
\end{bmatrix}.
\end{eqnarray} One of our main contributions in this work is to develop a preconditioned MINRES method with the proposed preconditioners, which has theoretically guaranteed convergence based on eigenvalues.

It is worth noting that our preconditioning approaches are fundamentally different from another kind of related existing work (see, e.g., \cite{pearson2012regularization,Linheatopt2022}), which is a typical preconditioning approach in the context of preconditioning for saddle point systems. Its effectiveness is based on the approximation of Schur complements (e.g., \cite{PearsonWathen2012,LevequePearson22}) and the classical preconditioning techniques \cite{AxelssonNeytcheva2006, MurphyGolubWathen2000}. Yet, for instance, our preconditioning proposal extends the MINRES preconditioning strategy proposed in \cite{hondongSC2023} from optimal control of wave equations to that of parabolic equations, resulting in a clustered spectrum around $\{\pm 1\}$. Moreover, the implementation of our preconditioners based on FFTs are parallel-in-time. 

The paper is organized as follows. In Section \ref{sec:prelim}, we review some preliminary results on block Toeplitz matrices. In Section \ref{sec:main}, we provide our main results on the spectral analysis for our proposed preconditioners. Numerical examples are given in Section \ref{sec:numerical} for supporting the performance of our proposed preconditioners.

\section{Preliminaries on Toeplitz matrices}\label{sec:prelim}

In this section, we provide some useful background knowledge regarding Toeplitz matrices.

We let $L^1([-\pi,\pi])$ be the Banach space of all functions that are Lebesgue integrable over $[-\pi,\pi]$ and periodically extended to the whole real line. The Toeplitz matrix generated by $f \in L^1([-\pi,\pi])$ is denoted by $T_{n}[f]$, namely, 

	\begin{eqnarray*}
	T_{n}[f]=\begin{bmatrix}{}
	a_0 & a_{-1} & \cdots & a_{-n+2} & a_{-n+1} \\
	a_1 & a_0 & a_{-1}   &  & a_{-n+2} \\
	\vdots & a_1 & a_0 & \ddots & \vdots \\
	a_{n-2} &  & \ddots & \ddots & a_{-1} \\
	a_{n-1} & a_{n-2} &\cdots & a_1 & a_0
	\end{bmatrix},
	\end{eqnarray*}
    where
	\begin{eqnarray*}
	a_{k}=\frac{1}{2\pi} \int_{-\pi} ^{\pi}f(\theta) e^{-\mathbf{i} k \theta } \,d\theta,\quad k=0,\pm1,\pm2,\dots
	\end{eqnarray*}
	are the Fourier coefficients of $f$. The function $f$ is called the \emph{generating function} of $T_n[f]$. If $f$ is complex-valued, then $T_n[f]$ is non-Hermitian for all sufficiently large $n$. Conversely, if $f$ is real-valued, then $T_n[f]$ is Hermitian for all $n$. If $f$ is real-valued and nonnegative, but not identically zero almost everywhere, then $T_n[f]$ is Hermitian positive definite for all $n$. If $f$ is real-valued and even, $T_n[f]$ is symmetric for all $n$. For thorough discussions on the related properties of block Toeplitz matrices, we refer readers to \cite{book-GLT-I} and  references therein; for computational features see \cite{MR2108963,Chan:1996:CGM:240441.240445,book-GLT-II} and references there reported.

\section{Main results}\label{sec:main}

In this section, the main results which support the effectiveness of our proposed preconditioners are provided. Also, the implementation issue is discussed.

\subsection{GMRES - block $\omega$-circulant based preconditioner}\label{sub:strang_preconditioner_gmres}

\begin{proposition}\label{thm:strang_preconditioner_gmres}
Let $\mathcal{\widehat{A}} \in \mathbb{R}^{2mn \times 2mn} ,\mathcal{P}_{S} \in \mathbb{C}^{2mn \times 2mn}$ be defined by (\ref{eqn:matrix_A_hat}) and (\ref{eqn:matrix_P_GMRES}), respectively. Then,
\[
\mathcal{P}_{S}^{-1} \mathcal{\widehat{A}} = I_{nm} + \widetilde{\mathcal{R}}_{1},
\]
where $I_{mn}$ is the $mn$ by $mn$ identity matrix and $\mathrm{rank}(\widetilde{\mathcal{R}}_{1}) \leq 4m$.
\end{proposition}

\begin{proof}
First, we observe that
    \begin{align*}
        \mathcal{P}_{S} - \mathcal{\widehat{A}} &= \mat{ \mathcal{S-T}& \\ &(\mathcal{S-T})^{*} } \\
        &= \mat{ (S_{n}-B_{n})\otimes I_{m} \\ &(S_{n}-B_{n})^{*}\otimes I_{m} }.
    \end{align*}
    
    Now, we examine $\mathrm{rank}(S_{n}-B_{n})$ via the following matrix decomposition:
\begin{eqnarray*}
    S_{n}-B_{n} &=& S_n^{(1)}(S_n^{(2)})^{-1} - B_n^{(1)}(B_n^{(2)})^{-1} \\
    &=& (S_n^{(1)} - B_n^{(1)}) (S_n^{(2)})^{-1} + B_n^{(1)}\big( (S_n^{(2)})^{-1} - (B_n^{(2)})^{-1} \big).
\end{eqnarray*}
From the simple structure of these matrices, it is clear that $\mathrm{rank}(S_n^{(1)} - B_n^{(1)}) \leq 1$ and $\mathrm{rank}\big( (S_n^{(2)})^{-1} - (B_n^{(2)})^{-1} \big) \leq 1$ (because $\mathrm{rank} (S_n^{(2)} - B_n^{(2)}) \leq 1$). Thus, we have $\mathrm{rank}( S_{n}-B_{n} ) \leq 2$, implying $\mathrm{rank}( \mathcal{P}_{S} - \mathcal{\widehat{A}}) \leq 4m$.
    Then, we have
    \begin{align*}
        \mathcal{P}_{S}^{-1} \mathcal{\widehat{A}} &= I_{nm} - \underbrace{\mathcal{P}_{S}^{-1} ( \mathcal{P}_{S} - \mathcal{\widehat{A}}  ) }_{=:\widetilde{\mathcal{R}}_{1}},
    \end{align*}
    where $\mathrm{rank}(\widetilde{\mathcal{R}}_{1}) \leq 4m$. 
\end{proof}


As a consequence of proposition \ref{thm:strang_preconditioner_gmres}, we can show that the singular values of $\mathcal{P}_{S}^{-1} \mathcal{\widehat{A}}$ are clustered around unity except for a number of outliers whose size is independent of $n$ in general. From a preconditioning for nonsymmetric Toeplitz systems point of view, such a singular value cluster is often used to support the preconditioning effectiveness of $\mathcal{P}_{S}$ for $\mathcal{\widehat{A}}$ when GMRES is used. We refer to \cite{MR2108963} for a systematic exposition of preconditioning for non-Hermitian Toeplitz systems. One could further show that the eigenvalues of $\mathcal{P}_{S}^{-1}\mathcal{\widehat{A}}$ are also clustered around unity, as with many existing works. However, as mentioned in Section \ref{sec:Introduction}, the convergence of GMRES in general cannot be rigorously analyzed by using only eigenvalues. As such, in the next subsections, we provide theoretical supports for our proposed MINRES solvers.

\subsection{MINRES - ideal preconditioner}\label{sub:ideal_preconditioner}

In what follows, we will show that an ideal preconditioner for $\mathcal{A}$ is the SPD matrix $|\mathcal{A}|$ defined by (\ref{eqn:abs_ideal_matrix_H}).

\begin{proposition}\label{proposition:ideal_precon}
Let $\mathcal{A} \in \mathbb{R}^{2mn \times 2mn}$ be defined by (\ref{eqn:matrix_A}). Then, the preconditioned matrix $|\mathcal{A}|^{-1}\mathcal{A}$ is both (real) symmetric and orthogonal.
\end{proposition}

\begin{proof}
Considering the singular value decomposition of $\mathcal{T} = U \Sigma V^{\top}$, the associated decomposition of $\mathcal{A}$ is obtained by direct computation, that is
\begin{eqnarray}\label{eqn:matrix_Mat}
\mathcal{A} &=&
\begin{bmatrix} 
\alpha {I}_n \otimes I_m  & V\Sigma U^{\top}\\
 U\Sigma V^{\top}  &  -\alpha{I}_n \otimes I_m
\end{bmatrix} \\\nonumber
&=&
\begin{bmatrix} 
V  & \\
   &  U
\end{bmatrix}
\begin{bmatrix} 
\alpha {I}_n \otimes I_m  & \Sigma\\
 \Sigma  &  -\alpha {I}_n \otimes I_m
\end{bmatrix}
\begin{bmatrix} 
V  & \\
   &  U
\end{bmatrix}^{\top}  \\\nonumber
&=&\begin{bmatrix} 
V  & \\
   &  U
\end{bmatrix}
\mathcal{{\widehat{Q}}}
\begin{bmatrix} 
\sqrt{\Sigma^2+\alpha^2 {I}_n \otimes I_m}  & \\
   &  -\sqrt{\Sigma^2+\alpha^2 {I}_n \otimes I_m}
\end{bmatrix}
\mathcal{{\widehat{Q}}}^{\top}
\begin{bmatrix} 
V  & \\
   &  U
\end{bmatrix}^{\top},
\end{eqnarray} 
where $\mathcal{{\widehat{Q}}}$ is orthogonal given by 
\begin{eqnarray*}
\mathcal{{\widehat{Q}}}
=
\begin{bmatrix} 
 \Sigma D_1^{-1}& -\Sigma D_2^{-1}\\
 \big(\sqrt{\Sigma^2+\alpha^2 {I}_n \otimes I_m}  - \alpha  {I}_n \otimes I_m\big)D_1^{-1}   & \big(\sqrt{\Sigma^2+\alpha^2 {I}_n \otimes I_m}  + \alpha {I}_n \otimes I_m\big)D_2^{-1}
\end{bmatrix}
\end{eqnarray*}
with
\[
D_1 =  \sqrt{(-\sqrt{\Sigma^2+\alpha^2 {I}_n \otimes I_m}  + \alpha {I}_n \otimes I_m)^2 + \Sigma^2} 
\]
and
\[
D_2 = \sqrt{(\sqrt{\Sigma^2+\alpha^2 {I}_n \otimes I_m}  + \alpha {I}_n \otimes I_m)^2 + \Sigma^2 }.
\]
It is obvious that both diagonal matrices $D_1$ and $D_2$ are invertible. Thus $\widehat{Q}$ is well-defined.

Since both $\widehat{Q}$ and $\begin{bmatrix} 
V  & \\
   &  U
\end{bmatrix}$ are orthogonal, $\mathcal{Q}:=\begin{bmatrix} 
V  & \\
   &  U
\end{bmatrix}\widehat{Q}$ is also orthogonal. Hence, from (\ref{eqn:matrix_Mat}), we have obtained an eigendecomposition of $\mathcal{A}$, i.e., 
\[
\mathcal{A}=\mathcal{Q}\begin{bmatrix} 
\sqrt{\Sigma^2+\alpha^2 {I}_n \otimes I_m}  & \\
   &  -\sqrt{\Sigma^2+\alpha^2 {I}_n \otimes I_m }
\end{bmatrix}
\mathcal{Q}^{\top}.
\]

Thus, we have
\begin{eqnarray*}
|\mathcal{A}|&=&\mathcal{Q} 
\begin{bmatrix} 
\sqrt{\Sigma^2 + \alpha^2 {I}_n \otimes I_m}  & \\
   &  \sqrt{\Sigma^2 + \alpha^2 {I}_n \otimes I_m }
\end{bmatrix}
  \mathcal{Q}^{\top} ,
  \end{eqnarray*}
  where $\mathcal{Q}$ is orthogonal and  $\Sigma$ is a diagonal matrix containing the singular values of $\mathcal{T}$.
Thus,
\begin{eqnarray*}
|\mathcal{A}|^{-1} \mathcal{A} &=&\mathcal{Q} 
\begin{bmatrix} 
{I}_n \otimes I_m  & \\
   &  - {I}_n \otimes I_m
\end{bmatrix}
  \mathcal{Q}^{\top},
\end{eqnarray*}
which is both symmetric and orthogonal. The proof is complete.
\end{proof}

In other words, Proposition \ref{proposition:ideal_precon} shows that the preconditioner $|\mathcal{A}|$ can render the eigenvalues exactly at $\pm 1$, providing a good guide to designing effective preconditioners for $\mathcal{A}$. As a consequence of Proposition \ref{proposition:ideal_precon}, we conclude that the MINRES with $|\mathcal{A}|$ as a preconditioner can achieve mesh-independent convergence, i.e., a convergence rate independent of both the meshes and the regularization parameter.

Despite the fact that $|\mathcal{A}|$ is an ideal preconditioner, its direct application has the drawback of being computationally expensive in general. Proposition \ref{proposition:ideal_precon} reveals an eigendecomposition of both $\mathcal{A}$ and $|\mathcal{A}|$, allowing us to develop preconditioners based on the spectral symbol. In what follows, we will show that $\mathcal{P}_{S}$ defined by (\ref{eqn:matrix_P}) is a good preconditioner for $\mathcal{A}$ in the sense that the preconditioned matrix $\mathcal{P}_{S}^{-1} \mathcal{A}$ can be expressed as the sum of a Hermitian unitary matrix and a low-rank matrix.

\subsection{MINRES - block $\omega$-circulant based preconditioner}\label{sub:strang_preconditioner}
The following theorem accounts for the preconditioning effect of MINRES-$\mathcal{P}_{S}$.
\begin{theorem}\label{thm:strang_preconditioner}
Let $\mathcal{A} \in \mathbb{R}^{2mn \times 2mn} ,|\mathcal{P}_{S}| \in \mathbb{C}^{2mn \times 2mn} $ be defined by (\ref{eqn:matrix_A}) and (\ref{eqn:matrix_P}), respectively. Then,
\[
|\mathcal{P}_{S}|^{-1} \mathcal{A} = \widetilde{\mathcal{Q}}_{1} + \widetilde{\mathcal{R}}_{2},
\]
where $\widetilde{\mathcal{Q}}_{1}$ is both Hermitian and unitary and $\mathrm{rank}(\widetilde{\mathcal{R}}_{2}) \leq 4m$.
\end{theorem}
\begin{proof}
    Let $s(\mathcal{A})=\mat{\alpha {I}_n \otimes I_m & \mathcal{S}^{*} \\ \mathcal{S} & -\alpha {I}_n \otimes I_m}$. Notice that $|s(\mathcal{A})| = \sqrt{s(\mathcal{A})^2} = |\mathcal{P}_{S}| = \sqrt{\mathcal{P}_{S}^*\mathcal{P}_{S} }$. \\
    Simple calculations show that
    \begin{align*}
        s(\mathcal{A}) - \mathcal{A} &= \mat{&(\mathcal{S-T})^{*} \\ \mathcal{S-T}&} \\
        &= \mat{&(S_{n}-B_{n})^{*}\otimes I_{m} \\ (S_{n}-B_{n})\otimes I_{m}}.
    \end{align*}
    Thus, $\mathrm{rank}( s(\mathcal{A}) - \mathcal{A}  ) \leq 4m$, following an argument from the proof of Proposition \ref{thm:strang_preconditioner_gmres}. Thus, we have
    \begin{align*}
        |\mathcal{P}_{S}|^{-1} \mathcal{A} &= |s(\mathcal{A})|^{-1} \mathcal{A}\\
        &= \underbrace{|s(\mathcal{A})|^{-1} s(\mathcal{A})}_{=:\widetilde{\mathcal{Q}}_{1}} - \underbrace{|s(\mathcal{A})|^{-1} ( s(\mathcal{A}) - \mathcal{A} ) }_{=:\widetilde{\mathcal{R}}_{2}},
    \end{align*}
    where $\mathrm{rank}(\widetilde{\mathcal{R}}_{2})  \leq 4m$. 

    Since $s(\mathcal{A})$ is Hermitian, we have $|s(\mathcal{A})|=\mathcal{W}\Psi\mathcal{W}^*$, where $\mathcal{W}$ is a unitary matrix and $\Psi$ is diagonal matrix containing the eigenvalues of $s(\mathcal{A})$. Correspondingly, we have $s(\mathcal{A})=\mathcal{W}|\Psi|\mathcal{W}^*$, where $|\Psi|$ is diagonal matrix containing the absolute value eigenvalues of $s(\mathcal{A})$. Thus, we have $\widetilde{\mathcal{Q}}_{1} =\mathcal{W}|\Psi|^{-1}\Psi\mathcal{W}^*$,  which is clearly Hermitian. Also, as $\widetilde{\mathcal{Q}}_{1}^2 =\mathcal{W}|\Psi|^{-1}\Psi|\Psi|^{-1}\Psi\mathcal{W}^*=\mathcal{W} I_{2nm}\mathcal{W}^*=I_{2nm}$, we know that $\widetilde{\mathcal{Q}}_{1}$ is unitary.

    The proof is complete.
\end{proof}

As a consequence of Theorem \ref{thm:strang_preconditioner} and \cite[Corollary 3]{BRANDTS20103100}, we know that the preconditioned matrix $|\mathcal{P}_{S}|^{-1} \mathcal{A}$ has clustered eigenvalues at $\pm 1$, with a number of outliers independent of $n$ in general (i.e., depending only on $m$). Thus, the convergence is independent of the time step in general, and we can expect that MINRES for $\mathcal{A}$ will converge rapidly in exact arithmetic with $|\mathcal{P}_{S}|$ as the preconditioner.

\subsection{MINRES - modified block $\omega$-circulant based preconditioner}\label{sub:Modified Strang}

To support the preconditioning effect of $\mathcal{P}_{MS}$, we will show that it is spectrally equivalent to $\mathcal{P}_S$. Before proceeding, we introduce the following auxiliary matrix $\mathcal{P}_{AS}$ which is useful to showing the preconditioning effect for $\mathcal{P}_{MS}$.
\begin{eqnarray}\label{eqn:matrix_P_AuxS}
&&\mathcal{P}_{AS}\\ \nonumber
&=&
\begin{bmatrix} 
\sqrt{(S_{n}^* S_n +  \alpha^2 I_{n} )\otimes I_{m} + \tau^2 I_{n} \otimes L_m^2}
&  \\
  & \sqrt{(S_n S_{n}^* +  \alpha^2 I_{n} )\otimes I_{m} + \tau^2 I_{n} \otimes L_m^2}
\end{bmatrix}.
\end{eqnarray}

Also, the following lemma is useful.

\begin{lemma}\label{lemma:S_symmetric_part}
Let $S_n\in \mathbb{C}^{n \times n}$ be defined in (\ref{eqn:matrix_S}). Then, $S_n^* + S_n$ is (Hermitian) positive semi-definite.
\end{lemma}
\begin{proof}
The eigenvalues of $S_n$ can be found explicitly, which are 
\begin{eqnarray*}
\lambda_k(S_n) = \frac{\lambda_k(S_n^{(1)})}{\lambda_k(S_n^{(2)})} = \frac{1 - \exp{\big(\mathbf{i}(\frac{\zeta + 2\pi k}{n})\big)}}{\theta + (1-\theta)\exp{\big(\mathbf{i}(\frac{\zeta + 2\pi k}{n})\big)}},    
\end{eqnarray*}
$ k = 0,\dots, n-1.$
Thus, we have 
\begin{eqnarray*}
\lambda_k(S_n^* + S_n) &=& \frac{1 - \exp{\big(-\mathbf{i}(\frac{\zeta + 2\pi k}{n})\big)}}{\theta + (1-\theta)\exp{\big(-\mathbf{i}(\frac{\zeta + 2\pi k}{n})\big)}} + \frac{1 - \exp{\big(\mathbf{i}(\frac{\zeta + 2\pi k}{n})\big)}}{\theta + (1-\theta)\exp{\big(\mathbf{i}(\frac{\zeta + 2\pi k}{n})\big)}}\\
&=&\frac{2(2\theta-1)\big(1-\cos{(\frac{\zeta + 2\pi k}{n})}\big)}{ \theta^2 + (1-\theta)^2 + 2\theta(1-\theta)\cos{(\frac{\zeta + 2\pi k}{n})}  },
\end{eqnarray*}
which is always non-negative provided $S_n^{(2)}$ is nonsingular by assumption. The proof is complete.
\end{proof}

Denote by $\sigma({\bf C})$ the spectrum of a square matrix ${\bf C}$.

\begin{lemma}\label{lemma:P_AuxS_one}
Let $X=S_n\otimes I_m \in \mathbb{C}^{mn\times mn}$ and $Y=\tau  I_n\otimes (-L_m) \in \mathbb{R}^{mn\times mn}$, with the involved notation defined in (\ref{eqn:matrix_S}). Then,
\begin{eqnarray*}
\sigma \bigg(\sqrt{X^*X+Y^*Y +  \alpha^2 {I}_n \otimes I_m }^{-1} \sqrt{(X+Y)^* (X+Y) +  \alpha^2 {I}_n \otimes I_m } \bigg) \subseteq [1, \sqrt{2}].
\end{eqnarray*}
\end{lemma}

\begin{proof}
Let \[Z=\sqrt{(X+Y)^* (X+Y) +  \alpha^2 {I}_n \otimes I_m }\] and \[\widetilde{Z}=\sqrt{X^*X+Y^*Y +  \alpha^2 {I}_n \otimes I_m }. \] Knowing that the eigendecomposition of $-L_{m}$ is given by $-L_{m} = \mathbb{U}_{m} \Omega_{m} \mathbb{U}_{m}^{\top}$ with $-L_m$ assumed SPD, where $\mathbb{U}_{m}$ is orthogonal and $ \Omega_{m}$ is a real-valued diagonal matrix containing the eigenvalues of $-L_{m}$, we have

\begin{eqnarray*}
    Z&=&\sqrt{(X+Y)^* (X+Y) +  \alpha^2 {I}_n \otimes I_m }\\
    &=& (\Gamma_n\mathbb{F}_{n}\otimes \mathbb{U}_{m}) \\
    &&\times \sqrt{ (\Lambda_{n}\otimes I_{m} + \tau I_{n} \otimes \Omega_{m})^*(\Lambda_{n}\otimes I_{m} + \tau I_{n} \otimes \Omega_{m})   +\alpha^2 {I}_n \otimes I_m}\\
    &&\times (\Gamma_n\mathbb{F}_{n}\otimes \mathbb{U}_{m})^{*}
\end{eqnarray*}
and
\begin{eqnarray*}
    \widetilde{Z}&=&\sqrt{X^*X+Y^*Y +  \alpha^2 {I}_n \otimes I_m }\\
    &=& (\Gamma_n\mathbb{F}_{n}\otimes \mathbb{U}_{m}) \\
    &&\times \sqrt{ (\Lambda_{n}\otimes I_{m})^* (\Lambda_{n}\otimes I_{m})+(\tau I_{n} \otimes \Omega_{m})^{\top}(\tau I_{n} \otimes \Omega_{m}) +  \alpha^2 {I}_n \otimes I_m } \\
    &&\times(\Gamma_n\mathbb{F}_{n}\otimes \mathbb{U}_{m})^{*}.
\end{eqnarray*}
Since $Z$ and $\widetilde{Z}$ are diagonalized by the unitary matrix $\mathbb{Q}=\Gamma_n\mathbb{F}_{n}\otimes \mathbb{U}_{m}$, they are simultaneously diagonalizable.

Since both $Z$ and $\widetilde{Z}$ are invertible, they are Hermitian positive definite by construction. To examine the target spectrum of $\widetilde{Z}^{-1}Z$, we consider the Rayleigh quotient for complex $\mathbf{v} \neq \mathbf{0}$:

\begin{eqnarray*}
R &:=\joinrel=\joinrel=\joinrel=\joinrel=\joinrel=& \frac{ \mathbf{v}^* \sqrt{(X+Y)^* (X+Y) +  \alpha^2 {I}_n \otimes I_m } \mathbf{v} }{ \mathbf{v}^* \sqrt{X^*X+Y^*Y +  \alpha^2 {I}_n \otimes I_m } \mathbf{v} }\\
&\underbrace{=\joinrel=\joinrel=\joinrel=\joinrel=\joinrel=}_{{\bf w}=\mathbb{Q}{\bf v}}& \frac{ \mathbf{w}^* \sqrt{ (\Lambda_{n}\otimes I_{m} + \tau I_{n} \otimes \Omega_{m})^*(\Lambda_{n}\otimes I_{m} + \tau I_{n} \otimes \Omega_{m})   +\alpha^2 {I}_n \otimes I_m}  \mathbf{w} }{ \mathbf{w}^* \sqrt{ (\Lambda_{n}\otimes I_{m})^* (\Lambda_{n}\otimes I_{m})+( \tau I_{n} \otimes \Omega_{m})^{\top}( \tau I_{n} \otimes \Omega_{m}) +  \alpha^2 {I}_n \otimes I_m } \mathbf{w} }.
\end{eqnarray*} By the invertibility of $Z$ and $\widetilde{Z}$, both numerator and denominator are positive. 

On one hand, we estimate an upper bound for $R$. Let $z_1$ and $z_2$ be an entry of $\Lambda_{n}\otimes I_{m}$ and $ \tau I_{n} \otimes \Omega_{m}$, respectively. We have
	\begin{align*}
		\frac{{1}}{2}(\overline{z_1 + z_2})(z_1+z_2) \leq {\bar{z}_1 {z_1}+\bar{z}_2 {z_2}} ,
	\end{align*}
 which implies
 	\begin{align*}
		\frac{{1}}{2}(\overline{z_1 + z_2})(z_1+z_2) + \alpha^2 \leq {\bar{z}_1 {z_1}+\bar{z}_2 {z_2}} + \alpha^2 ,
	\end{align*}
 since $\alpha^2=\frac{\tau^2}{{\gamma}}$ is positive.
 Thus, we have
 	\begin{align*}
		\frac{{1}}{\sqrt{2}}\sqrt{(\overline{z_1 + z_2})(z_1+z_2) + \alpha^2}\leq \sqrt{{\bar{z}_1 {z_1}+\bar{z}_2 {z_2}} + \alpha^2  }.
	\end{align*}
Therefore, 
 \begin{equation*}
   R \leq \sqrt{2}.
 \end{equation*}

On the other hand, we estimate an lower bound for $R$ by first examining the definiteness of the matrix $X^*Y + Y^*X$. Since $S_n^*+S_n$ is (Hermitian) non-negative definite by Lemma \ref{lemma:S_symmetric_part}, which implies
\begin{eqnarray*}
X^*Y + Y^*X &=& (S_n\otimes I_m)^{\top} (\tau  I_n\otimes (-L_m)) + (\tau  I_n\otimes (-L_m))^{\top}(S_n\otimes I_m ) \\
&=&\tau (S_n^* + S_n) \otimes (-L_m)
\end{eqnarray*} is also non-negative definite. Thus, we also have
\begin{eqnarray*}
  \sqrt{(\overline{z_1 + z_2})(z_1+z_2) + \alpha^2} &=& \sqrt{ \bar{z}_1 {z_1} + \bar{z}_2 {z_2} + \bar{z}_1 {z_2} + \bar{z}_2 {z_1} + \alpha^2}  \\
&\geq& \sqrt{{\bar{z}_1 {z_1}+\bar{z}_2 {z_2}} + \alpha^2  }, 
\end{eqnarray*} implying
  \begin{equation*}
   1 \leq R.
 \end{equation*}

 The proof is complete.
\end{proof}

Similarly, we can show the following lemma:

\begin{lemma}\label{lemma:P_AuxS_two}
Let $X=S_n\otimes I_m \in \mathbb{C}^{mn\times mn}$ and $Y=\tau  I_n\otimes (-L_m) \in \mathbb{R}^{mn\times mn}$, with the involved notation defined in (\ref{eqn:matrix_S}). Then,
\begin{eqnarray*}
\sigma \bigg(\sqrt{XX^*+YY^* +  \alpha^2 {I}_n \otimes I_m }^{-1} \sqrt{(X+Y)(X+Y)^* +  \alpha^2 {I}_n \otimes I_m } \bigg) \subseteq [1, \sqrt{2}].
\end{eqnarray*}
\end{lemma}

\begin{proposition}\label{proposition:P_AuxS}
Let $|\mathcal{P}_{S}| ,\mathcal{P}_{AS} \in \mathbb{C}^{2mn\times 2mn}$ be defined by (\ref{eqn:matrix_P}) and (\ref{eqn:matrix_P_AuxS}), respectively. Then,
\begin{eqnarray*}
    \sigma( \mathcal{P}_{AS}^{-1}|\mathcal{P}_{S}| ) \subseteq [1, \sqrt{2}].
\end{eqnarray*}
\end{proposition}

\begin{proof}
Knowing that
\[
|\mathcal{P}_{S}| = \mat{ \sqrt{(X+Y)^* (X+Y) +  \alpha^2 {I}_n \otimes I_m } &\\&\sqrt{(X+Y) (X+Y)^* +  \alpha^2 {I}_n \otimes I_m} }
\]
and
\[
\mathcal{P}_{AS} = \mat{ \sqrt{X^* X + Y^* Y +  \alpha^2 {I}_n \otimes I_m } &\\&\sqrt{XX^*+YY^* +  \alpha^2 {I}_n \otimes I_m } },
\]
we know that $\sigma( \mathcal{P}_{AS}^{-1}|\mathcal{P}_{S}| ) \subseteq [1, \sqrt{2}],$ by Lemmas \ref{lemma:P_AuxS_one} and \ref{lemma:P_AuxS_two}.

\end{proof}

\begin{remark}\label{remark:unity_for_CN}
    When the Crank-Nicolson method (i.e., $\theta = \frac{1}{2})$ is adopted, $S_n^* +S_n$ is a null matrix from Lemma \ref{lemma:S_symmetric_part}. Using this fact and considering the proofs of Lemmas \ref{lemma:P_AuxS_one} \& \ref{lemma:P_AuxS_two}, we can show that $X^*Y + Y^*X$ is also a null matrix which implies $\mathcal{P}_{AS} = |\mathcal{P}_{S}|$.
\end{remark}

\begin{lemma}\label{lemma:P_AuxS_three}
Let $X=S_n\otimes I_m \in \mathbb{C}^{mn\times mn}$ and $Y=\tau  I_n\otimes (-L_m) \in \mathbb{R}^{mn\times mn}$, with the involved notation defined in (\ref{eqn:matrix_S}). Then,
\begin{eqnarray*}
\sigma \bigg( \big(\sqrt{X^*X +  \alpha^2 {I}_n \otimes I_m} + \sqrt{Y^*Y} \big)^{-1} \sqrt{X^*X+Y^*Y +  \alpha^2 {I}_n \otimes I_m } \bigg) \subseteq \bigg[ \frac{1}{\sqrt{2}} , 1 \bigg] .
\end{eqnarray*}
\end{lemma}

\begin{proof}
Similar to the proof of Lemma \ref{lemma:P_AuxS_one}, we let \[\widehat{Z}=\sqrt{X^*X +  \alpha^2 {I}_n \otimes I_m} + \sqrt{Y^*Y}\] and \[\widetilde{Z}=\sqrt{X^*X+Y^*Y +  \alpha^2 {I}_n \otimes I_m }. \] Knowing that the eigendecomposition of $-L_{m}$ is given by $-L_{m} = \mathbb{U}_{m} \Omega_{m} \mathbb{U}_{m}^{\top}$, where $\mathbb{U}_{m}$ is orthogonal and $ \Omega_{m}$ is a real diagonal matrix containing the eigenvalues of $-L_{m}$, we have
\begin{eqnarray*}
    \widehat{Z}&=&\sqrt{X^*X +  \alpha^2 {I}_n \otimes I_m} + \sqrt{Y^*Y} \\
    &=& (\Gamma_n\mathbb{F}_{n}\otimes \mathbb{U}_{m}) \\
    &&\times \big( \sqrt{ (\Lambda_{n}\otimes I_{m})^* (\Lambda_{n}\otimes I_{m})+  \alpha^2 {I}_n \otimes I_m } +  \tau I_{n} \otimes \Omega_{m} \big)\\
    &&\times (\Gamma_n\mathbb{F}_{n}\otimes \mathbb{U}_{m})^{*}
\end{eqnarray*}
and, again, 
\begin{eqnarray*}
    \widetilde{Z}&=&\sqrt{X^*X+Y^*Y +  \alpha^2 {I}_n \otimes I_m }\\
    &=& (\Gamma_n\mathbb{F}_{n}\otimes \mathbb{U}_{m}) \\
    &&\times \sqrt{ (\Lambda_{n}\otimes I_{m})^* (\Lambda_{n}\otimes I_{m})+( \tau I_{n} \otimes \Omega_{m})^{\top}( \tau I_{n} \otimes \Omega_{m}) +  \alpha^2 {I}_n \otimes I_m } \\
    &&\times(\Gamma_n\mathbb{F}_{n}\otimes \mathbb{U}_{m})^{*}.
\end{eqnarray*}
Thus, $\widehat{Z}$ and $\widetilde{Z}$ are simultaneously diagonalized by the unitary matrix $\mathbb{Q}=\Gamma_n\mathbb{F}_{n}\otimes \mathbb{U}_{m}$.

Since both $\widehat{Z}$ and $\widetilde{Z}$ are invertible, they are Hermitian positive definite by construction. To examine the target spectrum of $\widehat{Z}^{-1}\widetilde{Z}$, we consider the Rayleigh quotient for complex $\mathbf{v} \neq \mathbf{0}$:

\begin{eqnarray*}
\widehat{R} &:=\joinrel=\joinrel=\joinrel=\joinrel=\joinrel=& \frac{ \mathbf{v}^* \sqrt{X^*X+Y^*Y +  \alpha^2 {I}_n \otimes I_m } \mathbf{v} }{ \mathbf{v}^*  \big(\sqrt{X^*X +  \alpha^2 {I}_n \otimes I_m} + \sqrt{Y^*Y} \big)   \mathbf{v} }\\
&\underbrace{=\joinrel=\joinrel=\joinrel=\joinrel=\joinrel=}_{{\bf w}=\mathbb{Q}{\bf v}}& \frac{ \mathbf{w}^* \sqrt{ (\Lambda_{n}\otimes I_{m})^* (\Lambda_{n}\otimes I_{m})+( \tau I_{n} \otimes \Omega_{m})^{\top}( \tau I_{n} \otimes \Omega_{m}) +  \alpha^2 {I}_n \otimes I_m }   \mathbf{w} }{ \mathbf{w}^* \big( \sqrt{ (\Lambda_{n}\otimes I_{m})^* (\Lambda_{n}\otimes I_{m})+  \alpha^2 {I}_n \otimes I_m } +  \tau I_{n} \otimes \Omega_{m} \big) \mathbf{w} }
\end{eqnarray*}

For two non-negative numbers, $c_1$ and $c_2$, it is known that
	\begin{align*}
		\frac{1}{\sqrt{2}}(c_1+c_2)\leq\sqrt{c_1^2+c_2^2}\leq c_1+c_2. 
	\end{align*}

Therefore, by letting $c_1$ and $c_2$ be an entry of $\sqrt{ (\Lambda_{n}\otimes I_{m})^* (\Lambda_{n}\otimes I_{m})+  \alpha^2 {I}_n \otimes I_m }$ and $  \tau I_{n} \otimes \Omega_{m}$, respectively, we have
 \begin{equation*}
     \frac{1}{\sqrt{2}}  \leq \widehat{R} \leq 1.
 \end{equation*}

 The proof is complete.
\end{proof}

Similarly, we can show the following lemma and proposition.

\begin{lemma}\label{lemma:P_AuxS_four}
Let $X=S_n\otimes I_m \in \mathbb{C}^{mn\times mn}$ and $Y=\tau  I_n\otimes (-L_m) \in \mathbb{R}^{mn\times mn}$, with the involved notation defined in (\ref{eqn:matrix_S}). Then,
\begin{eqnarray*}
\sigma \bigg( \big(\sqrt{X X^* +  \alpha^2 {I}_n \otimes I_m} + \sqrt{Y Y^*} \big)^{-1} \sqrt{X X^*+ YY^* +  \alpha^2 {I}_n \otimes I_m } \bigg) \subseteq \bigg[ \frac{1}{\sqrt{2}} , 1 \bigg] .
\end{eqnarray*}
\end{lemma}

 \begin{proposition}\label{proposition:P_AuxS_second}
Let $\mathcal{P}_{AS} ,\mathcal{P}_{MS} \in \mathbb{C}^{2mn\times 2mn}$ be defined by (\ref{eqn:matrix_P_AuxS}) and (\ref{eqn:matrix_P_MS}), respectively. Then,
\begin{eqnarray*}
    \sigma( \mathcal{P}_{MS}^{-1}\mathcal{P}_{AS} ) \subseteq \bigg[\frac{1}{\sqrt{2}}, 1\bigg].
\end{eqnarray*}
\end{proposition}
\begin{proof}
Knowing that
\[
\mathcal{P}_{MS} = \mat{ \sqrt{X^*X +  \alpha^2 {I}_n \otimes I_m} + \sqrt{Y^*Y} &\\& \sqrt{X X^* +  \alpha^2 {I}_n \otimes I_m} + \sqrt{Y Y^*} }
\]
and
\[
\mathcal{P}_{AS} = \mat{ \sqrt{X^* X + Y^* Y +  \alpha^2 {I}_n \otimes I_m } &\\&\sqrt{XX^*+YY^* +  \alpha^2 {I}_n \otimes I_m } },
\]
we know that $\sigma( \mathcal{P}_{MS}^{-1}\mathcal{P}_{AS} )  \subseteq \big[\frac{1}{\sqrt{2}}, 1 \big],$ by Lemmas \ref{lemma:P_AuxS_three} and \ref{lemma:P_AuxS_four}.
\end{proof}

Now, we are ready to show that $\mathcal{P}_{MS}$ and $\mathcal{P}_{S}$ are spectrally equivalent in the following theorem, which explains the effectiveness of $\mathcal{P}_{MS}$.

\begin{theorem}\label{theorem:P_AuxS}
Let $|\mathcal{P}_{S}| ,\mathcal{P}_{MS} \in \mathbb{C}^{2mn\times 2mn}$ be defined by (\ref{eqn:matrix_P}) and (\ref{eqn:matrix_P_MS}), respectively. Then,
\begin{eqnarray*}
    \sigma( \mathcal{P}_{MS}^{-1}|\mathcal{P}_{S}| ) \subseteq \bigg[\frac{1}{\sqrt{2}}, \sqrt{2}\bigg].
\end{eqnarray*}
\end{theorem}
\begin{proof}
Notice that for complex $\mathbf{v} \neq \mathbf{0}$,
 \begin{eqnarray*}
  \frac{ \mathbf{v}^* |\mathcal{P}_{S}| \mathbf{v} }{ \mathbf{v}^* \mathcal{P}_{MS} \mathbf{v} } = \frac{ \mathbf{v}^* |\mathcal{P}_{S}| \mathbf{v} }{ \mathbf{v}^* \mathcal{P}_{AS} \mathbf{v} } \cdot \frac{ \mathbf{v}^* \mathcal{P}_{AS} \mathbf{v} }{ \mathbf{v}^* \mathcal{P}_{MS} \mathbf{v} }.
 \end{eqnarray*}
 
 By Propositions \ref{proposition:P_AuxS} and \ref{proposition:P_AuxS_second}, we have
 \[
\frac{1}{\sqrt{2}}\leq \frac{ \mathbf{v}^* |\mathcal{P}_{S}| \mathbf{v} }{ \mathbf{v}^* \mathcal{P}_{MS} \mathbf{v} } \leq \sqrt{2}.
 \]
  The proof is complete.
\end{proof}

\begin{remark}\label{remark:P_AuxS}
   From Remark \ref{remark:unity_for_CN}, we know that $\sigma( \mathcal{P}_{MS}^{-1}|\mathcal{P}_{S}| ) \subseteq [\frac{1}{\sqrt{2}}, 1]$ when $\theta=\frac{1}{2}$.
\end{remark}

In what follows, we will justify the preconditioning effectiveness of $\mathcal{P}_{MS}$ for $\mathcal{A}$.

\begin{lemma}\cite[Theorem 4.5.9 (Ostrowski)]{horn_johnson_1990}\label{lemma:Ostrowski}
        Let $A_m,W_m$ be $m \times m$ matrices.  Suppose $A_m$ is Hermitian and $W_m$ is nonsingular. Let the eigenvalues of $A_m$ and $W_mW_m^*$ be arranged in an increasing order. For each $k=1,2,\dots,m,$ there exists a positive real number $\theta_k$ such that $\lambda_1(W_mW_m^*) \leq \theta_k \leq \lambda_m(W_mW_m^*)$ and
    \begin{equation*}
        \lambda_k(W_mA_mW_m^*) = \theta_k \lambda_k(A_m).
    \end{equation*}
\end{lemma}
As a consequence of Theorem \ref{theorem:P_AuxS}, Lemma \ref{lemma:Ostrowski}, and \cite[Corollary 3]{BRANDTS20103100}, we can show the following corollary accounting for the preconditioning effectiveness of $\mathcal{P}_{MS}$:

\begin{corollary}
Let $\mathcal{A} \in \mathbb{R}^{2mn \times 2mn} ,\mathcal{P}_{MS} \in \mathbb{C}^{2mn \times 2mn} $ be defined by (\ref{eqn:matrix_A}) and (\ref{eqn:matrix_P_MS}), respectively. Then, the eigenvalues of the matrix $\mathcal{P}_{MS}^{-1} \mathcal{A}$ are contained $[-\sqrt{2}, -\frac{1}{\sqrt{2}}] \cup [\frac{1}{\sqrt{2}}, \sqrt{2}]$, with a number of outliers independent of $n$ in general (i.e., depending only on $m$). 
\end{corollary}
\begin{proof}
    Note that
    \begin{eqnarray*}
    \mathcal{P}_{MS}^{-1/2} \mathcal{A} \mathcal{P}_{MS}^{-1/2} = \mathcal{P}_{MS}^{-1/2} |\mathcal{P}_{S}|^{1/2} |\mathcal{P}_{S}|^{-1/2} \mathcal{A} |\mathcal{P}_{S}|^{-1/2} |\mathcal{P}_{S}|^{1/2} \mathcal{P}_{MS}^{-1/2}.
    \end{eqnarray*} 
    From Lemma \ref{lemma:Ostrowski} and Theorem \ref{theorem:P_AuxS}, we know that, for each $k=1,2,\dots, 2mn$, there exists a positive real number $\theta_k$ such that \[
    \frac{1}{\sqrt{2}}\leq\lambda_{\min}(\mathcal{P}_{MS}^{-1/2} |\mathcal{P}_{S}| \mathcal{P}_{MS}^{-1/2}) \leq \theta_k \leq \lambda_{\max}(\mathcal{P}_{MS}^{-1/2} |\mathcal{P}_{S}| \mathcal{P}_{MS}^{-1/2}) \leq \sqrt{2}
    \] 
    and
    \begin{equation*}
        \lambda_k(\mathcal{P}_{MS}^{-1/2} \mathcal{A} \mathcal{P}_{MS}^{-1/2}) = \theta_k \lambda_k(|\mathcal{P}_{S}|^{-1/2} \mathcal{A} |\mathcal{P}_{S}|^{-1/2}).
    \end{equation*}
    Recalling from Theorem \ref{thm:strang_preconditioner} and \cite[Corollary 3]{BRANDTS20103100} that $\lambda_k(|\mathcal{P}_{S}|^{-1/2} \mathcal{A} |\mathcal{P}_{S}|^{-1/2})$ are either $\pm 1$ except for a number of outliers independent of $n$ in general, the proof is complete.
\end{proof}

\begin{remark}
    When the Crank-Nicolson method is used, we can show that the eigenvalues of the matrix $\mathcal{P}_{MS}^{-1} \mathcal{A}$ are contained $[-1, -\frac{1}{\sqrt{2}}] \cup [\frac{1}{\sqrt{2}}, 1]$, with a number of outliers independent of $n$ in general.
\end{remark}
In light of the last corollary, we can expect that MINRES for $\mathcal{A}$ will converge rapidly in exact arithmetic with $\mathcal{P}_{MS}$ as the preconditioner.

\subsection{Implementation}\label{sub:implementation}
We begin by discussing the computation of $\mathcal{\widehat{A}}\mathbf{v}$ (and $\mathcal{A}\mathbf{v}$) for any given vector $\mathbf{v}$. The computation of matrix-vector product $\mathcal{\widehat{A}}\mathbf{v}$ can be computed in $\mathcal{O}(mn\log{n})$ operations by using fast Fourier transforms, because of the fact that $\mathcal{\widehat{A}}$ contains two block (dense) Toeplitz matrices. The required storage is of $\mathcal{O}(mn)$. In the special case when $\theta=1$ for instance, the product $\mathcal{\widehat{A}}\mathbf{v}$ requires only linear complexity of $\mathcal{O}(mn)$ since $\mathcal{A}$ is a sparse matrix with a simple bi-diagonal Toeplitz matrix $B_n$.

In each GMRES iteration, the matrix-vector product $\mathcal{P}_{S}^{-1}\mathbf{v}$ for a given vector $\mathbf{v}$ needs to be computed. Since $\omega$-circulant matrices are diagonalizable by the product of a diagonal matrix and a discrete Fourier matrix $\mathbb{F}_{n} =\frac{1}{\sqrt{n}}[\theta_{n}^{(i-1)(j-1)}]_{i,j=1}^{n}\in \mathbb{C}^{n\times n}$ with $\theta_{n} = \exp{(\frac{2\pi \mathbf{i}}{n})}$, we can represent the matrix $S_n$ defined by (\ref{eqn:matrix_S}) using eigendecomposition $S_n=\Gamma_n\mathbb{F}_{n}\Lambda_{n} \mathbb{F}_{n}^{*}\Gamma_n^*$. Note that $\Lambda_{n}$ is a diagonal matrix.

Hence, we can decompose $\mathcal{P}_{S}$ from (\ref{eqn:matrix_P_GMRES}) as follows:
\begin{align*}
    \mathcal{P}_{S} &= \mat{\mathcal{S}  & -\alpha {I}_n \otimes I_m \\
        \alpha {I}_n \otimes I_m  & \mathcal{S}^*}\\
                    &= \widetilde{\mathcal{U}} \left( \underbrace{\mat{\Lambda_{n} & -\alpha I_{n}\\ \alpha I_{n}&\Lambda_{n}^{*}}}_{\mathcal{G}} \otimes I_{m} + \tau \mat{I_{n}&\\&I_{n}}\otimes (-L_{m})\right) \widetilde{\mathcal{U}}^{*},
\end{align*}
where $\widetilde{\mathcal{U}}=\mat{\Gamma_n\mathbb{F}_{n}\otimes I_{m}&\\&(\Gamma_n\mathbb{F}_{n}\otimes I_m)^{*}}$ is an unitary matrix. Note that the matrix $\mathcal{G}$ can be further decomposed using the following Lemma in \cite{WuZhou2020}.
\begin{lemma}\label{lemma:four_diag_decomposition}(\cite{WuZhou2020}, Lemma 2.3)
    Let $G_{1,2,3,4} \in \mathbb{C}^{n\times n}$ be four diagonal matrices and $\mathbf{G}=\mat{G_{1}&G_{2}\\G_{3}&G_{4}}$. Suppose $G_{2}$ and $G_{3}$ are invertible. Then, it holds that
    \begin{align*}
        \mathbf{G} = \mathbf{W} \mat{G_{1}+G_{2}M_{1}&\\&G_{4}+G_{3}M_{2}} \mathbf{W}^{-1}, \quad \mathbf{W} = \mat{I_{n}&M_{2}\\M_{1}&I_{n}},
    \end{align*}
    provided $\mathbf{W}$ is invertible, where $I_{n}  \in \mathbb{R}^{n\times n} $ is the identity matrix and
    \begin{align*}
        M_{1}&=\frac{1}{2} G_{2}^{-1} \left( G_{4}-G_{1}+\sqrt{(G_{4}-G_{1})^2+4G_{2}G_{3}}\right),\\
        M_{2}&=\frac{-1}{2} G_{3}^{-1} \left( G_{4}-G_{1}+\sqrt{(G_{4}-G_{1})^2+4G_{2}G_{3}}\right).
    \end{align*}
\end{lemma}

Applying Lemma \ref{lemma:four_diag_decomposition} to the matrix $\mathcal{G} = \mathcal{WDW}^{-1}$, we can further decompose $\mathcal{P}_{S}$ as follows:

\begin{align*}
    \mathcal{P}_{S} &= \widetilde{\mathcal{U}} \left( \mathcal{WDW}^{-1} \otimes I_{m} + \tau \mat{I_{n}&\\& I_{n}}\otimes (-L_{m}) \right) \widetilde{\mathcal{U}}^{*}\\
                    &= \mathcal{V} \left( \mathcal{D} \otimes I_{m} + \tau \mat{I_{n}&\\& I_{n}}\otimes (-L_{m}) \right)  \mathcal{V}^{*},\\
                    &= \mathcal{V} \mat{D_{1}\otimes I_{m} + \tau I_{n}\otimes (-L_{m})&\\&D_{2}\otimes I_{m}+ \tau I_{n}\otimes (-L_{m})}  \mathcal{V}^{*},
\end{align*}
where $\mathcal{V}=\widetilde{\mathcal{U}} \left( \mathcal{W}\otimes I_{m} \right)$, and the matrices $\mathcal{W}$  and $\mathcal{D} = \mat{D_{1}&\\&D_{2}}$ are explicitly known from Lemma \ref{lemma:four_diag_decomposition}.

Therefore, the computation of $\mathbf{w} = \mathcal{P}_{S}^{-1}\mathbf{v}$ can be implemented by the following three steps.

\begin{enumerate}[1.]
    \item $\textrm{Compute}~\widetilde{\mathbf{v}} = \mathcal{V}^{*}\mathbf{\mathbf{v}}$,
    \item $\textrm{Compute}~\widetilde{\mathbf{w}} =\mat{D_{1}\otimes I_{m} + \tau I_{n}\otimes (-L_{m})&\\&D_{2}\otimes I_{m}+ \tau I_{n}\otimes (-L_{m})}^{-1}\widetilde{\mathbf{v}}$,
    \item $\textrm{Compute}~\mathbf{w} = \mathcal{V}\widetilde{\mathbf{w}}$.
\end{enumerate}

Both Steps 1 and 3 can be computed by fast Fourier transformation in $\mathcal{O}(mn\log{n})$. In Step 2, the shifted Laplacian systems can be efficiently solved for instance by using the multigrid method. A detailed description of this highly effective implementation can be found in \cite{HeLiu2022} for example. 

In each MINRES iteration, we need to compute a matrix-vector product in the form of $|\mathcal{P}_{S}|^{-1}\mathbf{v}$ for some given vector $\mathbf{v}$. The eigendecomposition of $-L_{m}$ is given by $-L_{m} = \mathbb{U}_{m} \Omega_{m} \mathbb{U}_{m}^{\top}$ with $-L_m$ assumed SPD, where $\mathbb{U}_{m}$ is orthogonal and $ \Omega_{m}$ is a diagonal matrix containing the eigenvalues of $-L_{m}$.

Hence, we can rewrite $|\mathcal{P}_{S}|$ from (\ref{eqn:matrix_P}) as follows:

\begin{align*}
   |\mathcal{P}_{S}| &= \mat{\sqrt{\mathcal{S}^* \mathcal{S} +  \alpha^2 {I}_n \otimes I_m}&\\&\sqrt{\mathcal{S}^* \mathcal{S} +  \alpha^2 {I}_n \otimes I_m}}\\
    &= \mathcal{U} \mat{\sqrt{|\Lambda|^2+\alpha^2 {I}_n \otimes I_m}&\\&\sqrt{|\Lambda|^2+\alpha^2 {I}_n \otimes I_m}} \mathcal{U}^{*}, 
\end{align*}
where $\mathcal{U} := \mat{\Gamma_n\mathbb{F}_{n}\otimes \mathbb{U}_{m}&\\&(\Gamma_n\mathbb{F}_{n}\otimes \mathbb{U}_{m})^{*}}$ is an unitary matrix and $\Lambda := \Lambda_{n}\otimes I_{m} + \tau I_{n} \otimes \Omega_{m}$.

Therefore, the computation of $\mathbf{w} = |\mathcal{P}_{S}|^{-1}\mathbf{v}$ can be implemented with the following three steps.

\begin{enumerate}[1.]
    \item $\textrm{Compute}~\widetilde{\mathbf{v}} = \mathcal{U}^{*}\mathbf{\mathbf{v}}$,
    \item $\textrm{Compute}~\widetilde{\mathbf{w}} = \mat{\sqrt{|\Lambda|^2+\alpha^2 {I}_n \otimes I_m}^{-1}&\\&\sqrt{|\Lambda|^2+\alpha^2 {I}_n \otimes I_m}^{-1}}\widetilde{\mathbf{v}}$,
    \item $\textrm{Compute}~\mathbf{w} = \mathcal{U}\widetilde{\mathbf{w}}$.
\end{enumerate}

When the spatial grid is uniformly partitioned, the orthogonal matrix $\mathbb{U}_{m}$ becomes the discrete sine matrix $\mathbb{S}_{m}$. In this case, Step 1 and 3 can be computed efficiently by fast Fourier transform and fast sine transform in $\mathcal{O}(mn\log{n})$ operations. For step 2, the required computations take $\mathcal{O}(mn)$ operations since the matrix involved is a simple diagonal matrix.

The product of $\mathcal{P}_{MS}^{-1}\mathbf{v}$ for any vector $\mathbf{v}$ can be implemented following the above procedures.
Note that 
\begin{align*}
&\mathcal{P}_{MS} \\
&= \mat{\sqrt{S_{n}^* S_n + \alpha^2 I_{n}} \otimes I_{m} + \tau I_n \otimes (-L_m)
& \\
& \sqrt{S_n S_{n}^* +  \alpha^2 I_{n}} \otimes I_{m} + \tau  I_n \otimes (-L_m)}\\
&=
\widetilde{\mathcal{U}}\mat{\sqrt{|\Lambda_{n}|^2 + \alpha^2 I_{n}}\otimes I_{m} + \tau I_n \otimes (-L_m) &\\
&\sqrt{|\Lambda_{n}|^2 + \alpha^2 I_{n}} + \tau  I_n \otimes (-L_m)}\widetilde{\mathcal{U}}^{*},
\end{align*}
where $\widetilde{\mathcal{U}}=\mat{\Gamma_n\mathbb{F}_{n}\otimes I_{m}&\\&(\Gamma_n\mathbb{F}_{n}\otimes I_m)^{*}}$ is an unitary matrix.

The computation of $\mathcal{P}_{MS}^{-1}\mathbf{v}$ can be implemented by the following three steps.

\begin{enumerate}[1.]
    \item $\textrm{Compute}~\widetilde{\mathbf{v}} = \widetilde{\mathcal{U}}^{*}\mathbf{\mathbf{v}}$,
    \item $\textrm{Compute}$
    \begin{eqnarray*}
    &&\widetilde{\mathbf{w}} \\
    &=& \mat{(\sqrt{|\Lambda_{n}|^2 + \alpha^2 I_{n}}\otimes I_{m} + \tau  I_n\otimes (-L_m))^{-1} &\\
&(\sqrt{|\Lambda_{n}|^2 + \alpha^2 I_{n}}+ \tau  I_n\otimes (-L_m))^{-1}}\widetilde{\mathbf{v}},
    \end{eqnarray*}
    \item $\textrm{Compute}~\mathbf{w} = \widetilde{\mathcal{U}}\widetilde{\mathbf{w}}$.
\end{enumerate}
Both Steps 1 and 3 can be computed by fast Fourier transformation in $\mathcal{O}(mn\log{n})$. As for Step 2, again, the shifted Laplacian systems can be efficiently solved by the multigrid method. We refer to \cite{WuZhou2020} for more details regarding such efficient implementation.

\section{Numerical examples}\label{sec:numerical}
In this section, we provide several numerical results to show the performance of our proposed preconditioners. All numerical experiments are carried out using MATLAB 2022b on a PC with Intel i5-13600KF CPU 3.50GHz and 32 GB RAM. 

The CPU time in seconds is measured using MATLAB built-in functions $\bold{tic \backslash toc}$. All Steps $1 - 3$ in Section \ref{sub:implementation} are implemented by the functions $\bold{dst}$ and $\bold{fft}$ as discrete sine transform and fast Fourier transform respectively. All Krylov subspace solvers used are implemented using the built-in functions on MATLAB. We choose a zero initial guess and a stopping tolerance of $10^{-8}$ based on the reduction in relative residual norms for all Krylov subspace solvers tested unless otherwise indicated.

We adopt the notation MINRES-$|\mathcal{P}_{S}|$ and MINRES-$\mathcal{P}_{MS}$ to represent the MINRES solvers with $|\mathcal{P}_{S}|$ and $\mathcal{P}_{MS}$, respectively. Also, GMRES-$\mathcal{P}_{S}$ is used to represent the GMRES solver with the proposed preconditioner $\mathcal{P}_{S}$. We compare our proposed methods against the state-of-the-art solver proposed recently in \cite{Linheatopt2022} (denoted by PCG-$\mathcal{P}_{\epsilon}$), where an $\epsilon$-circulant preconditioner $\mathcal{P}_{\epsilon}$ was constructed. Note that we did not compare with the matching Schur complement preconditioners proposed in \cite{PearsonWathen2012,LevequePearson22}. It is expected that their effectiveness cannot surpass PCG-$\mathcal{P}_{\epsilon}$ as studied in the numerical tests carried out in \cite{Linheatopt2022}. 

In the related tables, we denote by 'Iter' the number of iterations for solving a linear system by an iterative solver within the given accuracy. Denote by 'DoF', the number of unknowns in a linear system. Let $p^*$ and $y^*$ denote the approximate solution to $p$ and $y$, respectively. Then, we define the error measure $e_h$ as
\begin{equation}\label{eqn:error_measure}
e_h = \left\| \begin{bmatrix} y^*\\ p^* \end{bmatrix} - \begin{bmatrix} y\\ p \end{bmatrix} \right\|_{L^{\infty}_{\tau}(L^{2}(\Omega))}.
\end{equation}
The time interval $[0,T]$ and the space are partitioned uniformly with the mesh step size $\tau=T/n=T/{h^{-1}}$ and $h=1/(m+1)$, respectively, where $h$ can be found in the related tables. Also, only $\zeta=\pi$ is used in the related tables for $\omega = e^{\textbf{i}\zeta}$ in the block $\omega$-circulant preconditioners. It is because, after conducting extensive trials, it was consistently observed that the preconditioner corresponding to $\zeta=\pi$ yielded the best results.

\begin{example}\label{ex:example2}
    In this example \cite{Linheatopt2022}, we consider the following two-dimensional problem of solving (\ref{eqn:Cost_functional_heat}), where $\Omega=(0,1)^2$, $T = 1$, $a(x_1,x_2)=1$, and
    \begin{align*}
        f(x_1,x_2,t)&=(2\pi^2 -1)e^{-t}\sin{(\pi x_1)}\sin{(\pi x_2)},\\
        g(x_1,x_2,t)&=e^{-t}\sin{(\pi x_1)}\sin{(\pi x_2)},
    \end{align*}
    The analytical solution of which is given by
    \begin{eqnarray*}
        y(x_1,x_2,t)&=e^{-t}\sin{(\pi x_1)}\sin{(\pi x_2)},\quad p=0.
    \end{eqnarray*}
\end{example}

To support the result of Theorem \ref{theorem:P_AuxS}, we display the eigenvalues of the matrix $\mathcal{P}_{MS}^{-1}|\mathcal{P}_{S}|$ for various values of $\gamma$ in Figure \ref{fig:invPMS_PS}. The illustration confirms that the eigenvalues consistently fall within the interval $[\frac{1}{\sqrt{2}}, 1]$, aligning with the expectations set forth in Remark \ref{remark:P_AuxS}. Furthermore, it is evident that as $\gamma$ diminishes, the eigenvalues of $\mathcal{P}_{MS}^{-1}|\mathcal{P}_{S}|$ exhibit increased clustering around one. This trend can be attributed to the fact that $\alpha = \frac{\tau}{\sqrt{\gamma}}$ grows larger when $\gamma$ is reduced, assuming the matrix size (or $\tau$) remains constant. As $\alpha$ becomes larger, it follows from their respective definitions that $\mathcal{P}_{MS}$ becomes closer to $|\mathcal{P}_{S}|$.

\begin{figure}
     \centering
     \begin{subfigure}[b]{0.48\textwidth}
         \centering
         \includegraphics[width=\textwidth]{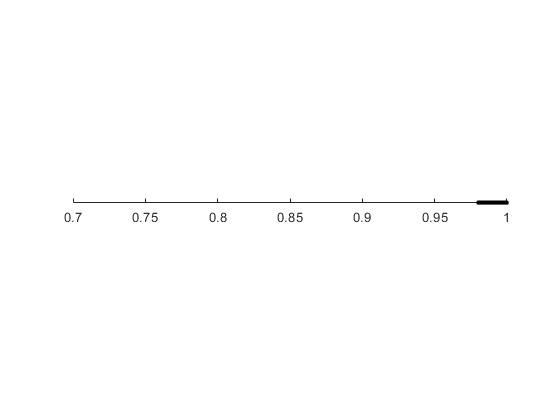}
         \caption{$\gamma=10^{-10}$}
     \end{subfigure}
     \hfill
     \begin{subfigure}[b]{0.48\textwidth}
         \centering
         \includegraphics[width=\textwidth]{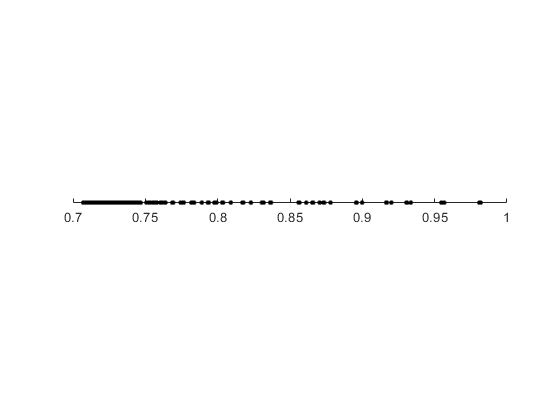}
         \caption{$\gamma=10^{-6}$}
     \end{subfigure}\\
     \begin{subfigure}[b]{0.48\textwidth}
         \centering
         \includegraphics[width=\textwidth]{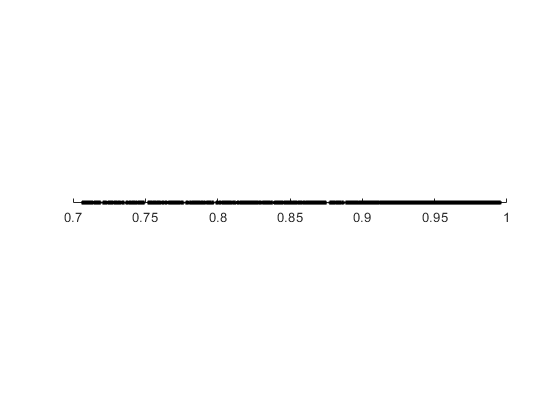}
         \caption{$\gamma=10^{-2}$}
     \end{subfigure}
 
        \caption{Eigenvalues of $\mathcal{P}_{MS}^{-1}|\mathcal{P}_{S}|$ with $n=16$, $m=15$, and various $\gamma$.}
        \label{fig:invPMS_PS}
\end{figure}

Table \ref{tab:table2_minres} displays the iteration counts, CPU times, and errors for GMRES-$\mathcal{P}_S$ and MINRES-$|\mathcal{P}_S|$ when applying the Crank-Nicolson method with different values of $\gamma$. Note that MINRES-$\mathcal{P}_{MS}$ was not implemented for this example, as $|\mathcal{P}_S|$ can already be efficiently implemented using fast sine transforms. We observe that: (i) both GMRES-$\mathcal{P}_S$ and MINRES-$|\mathcal{P}_S|$ with $\zeta=\pi$ perform excellently and stably, considering both iteration counts and CPU times across various values of $\gamma$; and (ii) the error decreases as the mesh is refined, except in the case when $\gamma = 10^{-10}$. In such an instance, the error exhibits only a slight decrease as the matrix size grows, which is likely due to the convergence tolerance used for MINRES not being sufficiently small to demonstrate the anticipated reduction. 

In Table \ref{tab:table2_pcgalpha}, we compare our preconditioners against PCG-$\mathcal{P}_{\epsilon}$ from \cite{Linheatopt2022} with $\epsilon = \frac{1}{2} \min \{\frac{\tau}{24\sqrt{\gamma}}, \frac{\tau^{\frac{3}{2}}}{2\sqrt{6\gamma}T}, \frac{\tau^2}{8\sqrt{3\gamma}T}, \frac{1}{3}\}$, where only the Crank-Nicolson method is considered. We report that for a larger value of $\gamma \geq 10^{-6}$, our proposed GMRES-$\mathcal{P}_{S}$ with $\omega=-1$ outperforms PCG-$\mathcal{P}_{\epsilon}$ significantly, namely, the computational time PCG-$\mathcal{P}_{\epsilon}$ needed for convergence is roughly two times larger. When $\gamma$ is small, GMRES-$\mathcal{P}_{S}$ is still highly comparable with PCG-$\mathcal{P}_{\epsilon}$ in terms of CPU times. Overall, GMRES-$\mathcal{P}_{S}$ is stable and robust in both iteration numbers and computational time for a wide range of $\gamma$.

\begin{example}\label{ex:example3}
    In this example, we consider the following two-dimensional problem of solving (\ref{eqn:Cost_functional_heat}) with a variable function $a(x_1,x_2)$, where $\Omega=(0,1)^2$, $T=1$, $a(x_1,x_2)=10^{-5} \sin{(\pi x_1 x_2)}$, and
    \begin{multline*}
        f(x_1,x_2,t)= - \sin{(\pi t)}\sin{(\pi x_1)}\sin{(\pi x_2)} + e^{-t} x_1 (1-x_1) [2\times 10^{-5}\sin{(\pi x_1 x_2)}\\
        - x_2 (1-x_2) - 10^{-5} \pi \cos{(\pi x_1 x_2)} x (1-2x_2)]\\
        +e^{-t} x_2 (1-x_2) [2\times 10^{-5} \sin{(\pi x_1 x_2)} - 10^{-5} \pi \cos{(\pi x_1 x_2)} x_2 (1-2x_1)],
    \end{multline*}
    \begin{multline*}
        g(x_1,x_2,t)=-\gamma \pi \cos{(\pi t)}\sin{(\pi x_1)}\sin{(\pi x_2)} + e^{-t} x_1 (1-x_1) x_2 (1-x_2)\\
        - 10^{-5}\gamma \pi^2 \sin{(\pi t)}[-2\sin{(\pi x_1 x_2)}\sin{(\pi x_1)}\sin{(\pi x_2)} \\
        + \cos{(\pi x_1 x_2)}(x_1\sin{(\pi x_1)}\cos{(\pi x_2)}+x_2\cos{(\pi x_1)}\sin{(\pi x_2)})].
    \end{multline*}
    The analytical solution of which is given by
    \begin{eqnarray*}
        y(x_1,x_2,t)=e^{-t} x_1 (1-x_1) x_2 (1-x_2),\quad p(x_1,x_2,t)=\gamma\sin{(\pi t)}\sin{(\pi x_1)}\sin{(\pi x_2)}.
    \end{eqnarray*}
\end{example}

In the given example, the direct application of MINRES-$|\mathcal{P}_{S}|$ is not feasible due to the non-diagonalizability of $-L_m$ using fast transform methods. Consequently, we adopt MINRES-$\mathcal{P}_{MS}$ which incorporates a multigrid method. Specifically, to solve the shifted Laplacian linear system (as detailed in Subsection \ref{sub:implementation}) and compute $\mathcal{P}_{MS}^{-1}\mathbf{v}$ for any vector $\mathbf{v}$, we apply one iteration of the V-cycle geometric multigrid method. In this iteration, the Gauss-Seidel method is employed as the pre-smoother.

Table \ref{tab:table3} shows the iteration numbers, CPU time, and error of GMRES-$\mathcal{P}_{S}$ and MINRES-$\mathcal{P}_{MS}$, respectively, when the Crank-Nicolson method is applied with a range of $\gamma$. This example aims to investigate the effectiveness of our solvers when $a(x_1,x_2)$ in (\ref{eqn:heat_2}) is not a constant. The results show that (i) GMRES-$\mathcal{P}_{S}$ with $\omega=-1$ maintains relatively stable iteration numbers and CPU time across a wide range of $\gamma$; (ii) MINRES-$\mathcal{P}_{MS}$ performs well for small $\gamma$, but its efficiency decreases as $\gamma$ increases — a phenomenon that has also been observed and reported in a previous study \cite{hondongSC2023}. This behavior may be attributed to the eigenvalue distribution of $\mathcal{A}$. Specifically, when $\gamma$ is very small compared to $\tau$, it is plausible to assume that $\alpha = \frac{\tau}{\sqrt{\gamma}}$ is quite large. Indeed, as $\gamma$ approaches $0^+$, it is corroborated by \cite[Corollary 3.2]{hondongSC2023} that the matrix-sequence $\left\{{\mathcal{A}\over \alpha} \right\}_n$ has eigenvalues relatively clustered around $\pm 1$, thus facilitating the solving of the all-at-once system. Additionally, as discussed in the preceding example, the increasing size of $\alpha$ results in $\mathcal{P}_{MS}$ closely resembling $|\mathcal{P}_{S}|$, which in turn leads to an improved preconditioning effect.

\begin{example}\label{ex:example4}
This example aims to test the robustness of our proposed method with the (homogeneous) Neumann boundary condition. We consider the following two-dimensional problem, where $\Omega=(0,1)^2$, $T = 1$, $a(x_1,x_2)=10^{-3}$, and
\begin{align*}
     f(x_1,x_2,t)&=(10^{-3} \times 8\pi^2 -1)e^{-t}\cos{(2\pi x_1)}\cos{(2\pi x_2)},\\
     g(x_1,x_2,t)&=e^{-t}\cos{(2\pi x_1)}\cos{(2\pi x_2)},
\end{align*}
    The analytical solution of which is given by
\begin{eqnarray*}
    y(x_1,x_2,t)&=e^{-t}\cos{(2\pi x_1)}\cos{(2\pi x_2)},\quad p=0.
\end{eqnarray*}
\end{example} Note that we use one iteration of the V-cycle geometric multigrid method with the Gauss-Seidel method as a pre-smoother to solve the shifted Laplacian linear system for GMRES-$\mathcal{P}_{S}$. Also, MINRES is not applicable in this case with the Neumann boundary condition, since $\mathcal{A}$ is not symmetric. Thus, we resort to using GMRES-$\mathcal{P}_{S}$.

Table \ref{tab:table4} shows the iteration numbers, CPU time, and error of GMRES-$\mathcal{P}_{S}$ when the Crank-Nicolson method is applied with various values of $\gamma$. The results indicate that GMRES-$\mathcal{P}_{S}$ maintains stable and low iteration numbers across a wide range of $\gamma$.

\section{Conclusions}\label{sec:conclusion}
In this work, we have provided a unifying preconditioning framework for circulant-based preconditioning applied to the concerned parabolic control problem. The framework is applicable for both first order (i.e., $\theta=1$) and second order (i.e., $\theta=1/2$) time discretization schemes. Moreover, it encompasses both circulant (i.e., $\omega=1$) and skew-circulant (i.e., $\omega=-1$ and $\theta=1$) preconditioners previously proposed in existing works. We note that it appears feasible to extend our proposed preconditioning theory to various implicit time-discretization methods, such as Backward Difference Formulas, as long as the block Toeplitz structure within the resulting all-at-once linear system remains intact. When considering more general discretization schemes, such as (semi-)implicit Runge-Kutta methods, a good starting point could be the study recently introduced in \cite{LevequeBergamaschi_etc_23}. This work develops a robust preconditioning approach designed specifically for all-at-once linear systems that are derived from the Runge-Kutta time discretization of time-dependent PDEs.

Specifically, we have proposed a class of block $\omega$-circulant based preconditioners for the all-at-once system of the parabolic control problem. First, when GMRES is considered, we have proposed a PinT preconditioner $\mathcal{P}_S$ for the concerned system. Second, when MINRES is used for the symmetrized system, we have constructed an ideal preconditioner $| \mathcal{A} |$, which can be used as a prototype for designing efficient preconditioners based on $| \mathcal{A} |$. Then, we have designed two novel preconditioners $|\mathcal{P}_{S}|$ and $\mathcal{P}_{MS}$ for the same problem, which can be efficiently implemented in a PinT manner. All proposed preconditioners have been shown effective in both numerical tests and a theoretical study. Based on our numerical tests, it has been demonstrated that our proposed solver, GMRES-$\mathcal{P}_S$ with $\omega=-1$ and $\theta=1/2$, can achieve rapid convergence, consistently maintaining stable iteration counts across a wide range of $\gamma$ values.   

We stress that the development of our proposed MINRES approach for optimal control problems is still in its infancy. As future work, we plan at least to develop more efficient preconditioned Krylov subspace solvers by integrating with an $\epsilon$-circulant matrix, where a small $\epsilon >0$ is chosen. In recent years, this approach has been shown successful for solving various PDEs (see, e.g, \cite{doi:10.1137/20M1316354, doi:10.1137/19M1309869, Sun2022, 2020arXiv200509158G}), achieving clustered singular values without any outliers. We will investigate whether such a combination can reduce the number of singular values/eigenvalue outliers that are present as a result from our preconditioners, which could achieve parameter-independent convergence in the MINRES framework.



\section*{Acknowledgments}
The work of Sean Hon was supported in part by the Hong Kong RGC under grant 22300921 and a start-up grant from the Croucher Foundation.

\begin{table}[!tbp]
\caption{Results of GMRES-$\mathcal{P}_{S}$ and MINRES-$|\mathcal{P}_{S}|$ for Example \ref{ex:example2} with $\zeta=\pi$ and $\theta=\frac{1}{2}$ (Crank-Nicolson)}
\label{tab:table2_minres}
\centering
\begin{tabular}{@{}clclclcccccc@{}}
\toprule
\multicolumn{2}{c}{\multirow{2}{*}{$\gamma$}} & \multicolumn{2}{c}{\multirow{2}{*}{$h$}} & \multicolumn{2}{c}{\multirow{2}{*}{DoF}} & \multicolumn{3}{c}{GMRES-$\mathcal{P}_{S}$} & \multicolumn{3}{c}{MINRES-$|\mathcal{P}_{S}|$} \\ \cmidrule(l){7-12} 
\multicolumn{2}{c}{} & \multicolumn{2}{c}{} & \multicolumn{2}{c}{} & Iter & CPU & $e_{h}$ & Iter & CPU & $e_{h}$ \\ \midrule
\multicolumn{2}{c}{\multirow{4}{*}{$10^{-10}$}} & \multicolumn{2}{c}{$2^{-5}$} & \multicolumn{2}{c}{61504} & 3 & 0.039 & 1.18e-9 & 3 & 0.033 & 3.18e-9 \\
\multicolumn{2}{c}{} & \multicolumn{2}{c}{$2^{-6}$} & \multicolumn{2}{c}{508032} & 3 & 0.35 & 1.18e-9 & 5 & 0.48 & 1.45e-9 \\
\multicolumn{2}{c}{} & \multicolumn{2}{c}{$2^{-7}$} & \multicolumn{2}{c}{4129024} & 3 & 3.32 & 1.04e-9 & 6 & 4.91 & 1.04e-9 \\
\multicolumn{2}{c}{} & \multicolumn{2}{c}{$2^{-8}$} & \multicolumn{2}{c}{33292800} & 3 & 27.96 & 4.49e-10 & 6 & 40.91 & 4.49e-10 \\ \midrule
\multicolumn{2}{c}{\multirow{4}{*}{$10^{-8}$}} & \multicolumn{2}{c}{$2^{-5}$} & \multicolumn{2}{c}{61504} & 3 & 0.030 & 1.12e-7 & 6 & 0.045 & 1.26e-7 \\
\multicolumn{2}{c}{} & \multicolumn{2}{c}{$2^{-6}$} & \multicolumn{2}{c}{508032} & 3 & 0.37 & 6.71e-8 & 6 & 0.52 & 6.71e-8 \\
\multicolumn{2}{c}{} & \multicolumn{2}{c}{$2^{-7}$} & \multicolumn{2}{c}{4129024} & 3 & 3.30 & 1.81e-8 & 6 & 4.90 & 1.81e-8 \\
\multicolumn{2}{c}{} & \multicolumn{2}{c}{$2^{-8}$} & \multicolumn{2}{c}{33292800} & 3 & 27.85 & 4.53e-9 & 6 & 40.83 & 4.53e-9 \\ \midrule
\multicolumn{2}{c}{\multirow{4}{*}{$10^{-6}$}} & \multicolumn{2}{c}{$2^{-5}$} & \multicolumn{2}{c}{61504} & 3 & 0.029 & 2.90e-6 & 6 & 0.044 & 2.90e-6 \\
\multicolumn{2}{c}{} & \multicolumn{2}{c}{$2^{-6}$} & \multicolumn{2}{c}{508032} & 3 & 0.34 & 7.26e-7 & 6 & 0.56 & 7.26e-7 \\
\multicolumn{2}{c}{} & \multicolumn{2}{c}{$2^{-7}$} & \multicolumn{2}{c}{4129024} & 3 & 3.35 & 1.81e-7 & 6 & 4.91 & 1.81e-7 \\
\multicolumn{2}{c}{} & \multicolumn{2}{c}{$2^{-8}$} & \multicolumn{2}{c}{33292800} & 3 & 27.76 & 4.54e-8 & 6 & 40.98 & 4.54e-8 \\ \midrule
\multicolumn{2}{c}{\multirow{4}{*}{$10^{-4}$}} & \multicolumn{2}{c}{$2^{-5}$} & \multicolumn{2}{c}{61504} & 3 & 0.029 & 2.87e-5 & 6 & 0.045 & 2.87e-5 \\
\multicolumn{2}{c}{} & \multicolumn{2}{c}{$2^{-6}$} & \multicolumn{2}{c}{508032} & 3 & 0.39 & 7.19e-6 & 6 & 0.51 & 7.19e-6 \\
\multicolumn{2}{c}{} & \multicolumn{2}{c}{$2^{-7}$} & \multicolumn{2}{c}{4129024} & 3 & 3.34 & 1.80e-6 & 6 & 4.91 & 1.80e-6 \\
\multicolumn{2}{c}{} & \multicolumn{2}{c}{$2^{-8}$} & \multicolumn{2}{c}{33292800} & 3 & 27.96 & 4.49e-7 & 6 & 40.84 & 4.49e-7 \\ \midrule
\multicolumn{2}{c}{\multirow{4}{*}{$10^{-2}$}} & \multicolumn{2}{c}{$2^{-5}$} & \multicolumn{2}{c}{61504} & 3 & 0.030 & 2.77e-4 & 6 & 0.046 & 2.77e-4 \\
\multicolumn{2}{c}{} & \multicolumn{2}{c}{$2^{-6}$} & \multicolumn{2}{c}{508032} & 3 & 0.37 & 6.91e-5 & 6 & 0.55 & 6.91e-5 \\
\multicolumn{2}{c}{} & \multicolumn{2}{c}{$2^{-7}$} & \multicolumn{2}{c}{4129024} & 3 & 3.35 & 1.73e-5 & 6 & 4.90 & 1.73e-5 \\
\multicolumn{2}{c}{} & \multicolumn{2}{c}{$2^{-8}$} & \multicolumn{2}{c}{33292800} & 3 & 27.80 & 4.31e-6 & 6 & 40.96 & 4.31e-6 \\ \bottomrule
\end{tabular}
\end{table}

\begin{table}[!tbp]
\caption{Results of PCG-$\mathcal{P}_{\epsilon}$ for Example \ref{ex:example2} with Crank-Nicolson method ($\theta=\frac{1}{2}$)}
\label{tab:table2_pcgalpha}
\centering
\begin{tabular}{@{}clclclccc@{}}
\toprule
\multicolumn{2}{c}{\multirow{2}{*}{$\gamma$}} & \multicolumn{2}{c}{\multirow{2}{*}{$h$}} & \multicolumn{2}{c}{\multirow{2}{*}{DoF}} & \multicolumn{3}{c}{PCG-$\mathcal{P}_{\epsilon}$} \\ \cmidrule(l){7-9} 
\multicolumn{2}{c}{} & \multicolumn{2}{c}{} & \multicolumn{2}{c}{} & Iter & CPU & $e_{h}$ \\ \midrule
\multicolumn{2}{c}{\multirow{4}{*}{$10^{-10}$}} & \multicolumn{2}{c}{$2^{-5}$} & \multicolumn{2}{c}{61504} & 3 & 0.084 & 1.18e-9 \\
\multicolumn{2}{c}{} & \multicolumn{2}{c}{$2^{-6}$} & \multicolumn{2}{c}{508032} & 3 & 0.23 & 1.18e-9 \\
\multicolumn{2}{c}{} & \multicolumn{2}{c}{$2^{-7}$} & \multicolumn{2}{c}{4129024} & 4 & 2.49 & 1.04e-9 \\
\multicolumn{2}{c}{} & \multicolumn{2}{c}{$2^{-8}$} & \multicolumn{2}{c}{33292800} & 5 & 24.55 & 4.49e-10 \\ \midrule
\multicolumn{2}{c}{\multirow{4}{*}{$10^{-8}$}} & \multicolumn{2}{c}{$2^{-5}$} & \multicolumn{2}{c}{61504} & 4 & 0.023 & 1.12e-7 \\
\multicolumn{2}{c}{} & \multicolumn{2}{c}{$2^{-6}$} & \multicolumn{2}{c}{508032} & 5 & 0.34 & 6.71e-8 \\
\multicolumn{2}{c}{} & \multicolumn{2}{c}{$2^{-7}$} & \multicolumn{2}{c}{4129024} & 5 & 3.01 & 1.81e-8 \\
\multicolumn{2}{c}{} & \multicolumn{2}{c}{$2^{-8}$} & \multicolumn{2}{c}{33292800} & 5 & 24.70 & 4.53e-9 \\ \midrule
\multicolumn{2}{c}{\multirow{4}{*}{$10^{-6}$}} & \multicolumn{2}{c}{$2^{-5}$} & \multicolumn{2}{c}{61504} & 5 & 0.046 & 2.90e-6 \\
\multicolumn{2}{c}{} & \multicolumn{2}{c}{$2^{-6}$} & \multicolumn{2}{c}{508032} & 5 & 0.33 & 7.26e-7 \\
\multicolumn{2}{c}{} & \multicolumn{2}{c}{$2^{-7}$} & \multicolumn{2}{c}{4129024} & 6 & 3.50 & 1.81e-7 \\
\multicolumn{2}{c}{} & \multicolumn{2}{c}{$2^{-8}$} & \multicolumn{2}{c}{33292800} & 6 & 28.64 & 4.54e-8 \\ \midrule
\multicolumn{2}{c}{\multirow{4}{*}{$10^{-4}$}} & \multicolumn{2}{c}{$2^{-5}$} & \multicolumn{2}{c}{61504} & 9 & 0.063 & 2.87e-5 \\
\multicolumn{2}{c}{} & \multicolumn{2}{c}{$2^{-6}$} & \multicolumn{2}{c}{508032} & 9 & 0.57 & 7.19e-6 \\
\multicolumn{2}{c}{} & \multicolumn{2}{c}{$2^{-7}$} & \multicolumn{2}{c}{4129024} & 9 & 4.95 & 1.80e-6 \\
\multicolumn{2}{c}{} & \multicolumn{2}{c}{$2^{-8}$} & \multicolumn{2}{c}{33292800} & 10 & 44.46 & 4.49e-7 \\ \midrule
\multicolumn{2}{c}{\multirow{4}{*}{$10^{-2}$}} & \multicolumn{2}{c}{$2^{-5}$} & \multicolumn{2}{c}{61504} & 11 & 0.061 & 2.77e-4 \\
\multicolumn{2}{c}{} & \multicolumn{2}{c}{$2^{-6}$} & \multicolumn{2}{c}{508032} & 11 & 0.68 & 6.91e-5 \\
\multicolumn{2}{c}{} & \multicolumn{2}{c}{$2^{-7}$} & \multicolumn{2}{c}{4129024} & 11 & 5.95 & 1.73e-5 \\
\multicolumn{2}{c}{} & \multicolumn{2}{c}{$2^{-8}$} & \multicolumn{2}{c}{33292800} & 12 & 53.96 & 4.31e-6 \\ \bottomrule
\end{tabular}
\end{table}

\begin{table}[!tbp]
\caption{Results of GMRES-$\mathcal{P}_{S}$ and MINRES-$\mathcal{P}_{MS}$ for Example \ref{ex:example3} with $\zeta=\pi$ and $\theta=\frac{1}{2}$ (Crank-Nicolson)}
\label{tab:table3}
\centering
\begin{tabular}{@{}clclclcccccc@{}}
\toprule
\multicolumn{2}{c}{\multirow{2}{*}{$\gamma$}} & \multicolumn{2}{c}{\multirow{2}{*}{$h$}} & \multicolumn{2}{c}{\multirow{2}{*}{DoF}} & \multicolumn{3}{c}{GMRES-$\mathcal{P}_{S}$} & \multicolumn{3}{c}{MINRES-$\mathcal{P}_{MS}$} \\ \cmidrule(l){7-12} 
\multicolumn{2}{c}{} & \multicolumn{2}{c}{} & \multicolumn{2}{c}{} & Iter & CPU & $e_{h}$ & Iter & CPU & $e_{h}$ \\ \midrule
\multicolumn{2}{c}{\multirow{4}{*}{$10^{-10}$}} & \multicolumn{2}{c}{$2^{-5}$} & \multicolumn{2}{c}{61504} & 3 & 0.14 & 6.60e-13 & 3 & 0.11 & 2.91e-10 \\
\multicolumn{2}{c}{} & \multicolumn{2}{c}{$2^{-6}$} & \multicolumn{2}{c}{508032} & 3 & 0.67 & 4.68e-13 & 5 & 0.74 & 2.55e-11 \\
\multicolumn{2}{c}{} & \multicolumn{2}{c}{$2^{-7}$} & \multicolumn{2}{c}{4129024} & 3 & 5.35 & 7.51e-13 & 6 & 6.27 & 4.43e-13 \\
\multicolumn{2}{c}{} & \multicolumn{2}{c}{$2^{-8}$} & \multicolumn{2}{c}{33292800} & 3 & 48.26 & 8.40e-12 & 6 & 61.63 & 1.73e-11 \\ \midrule
\multicolumn{2}{c}{\multirow{4}{*}{$10^{-8}$}} & \multicolumn{2}{c}{$2^{-5}$} & \multicolumn{2}{c}{61504} & 3 & 0.13 & 6.30e-11 & 6 & 0.17 & 6.29e-11 \\
\multicolumn{2}{c}{} & \multicolumn{2}{c}{$2^{-6}$} & \multicolumn{2}{c}{508032} & 3 & 0.64 & 5.48e-11 & 6 & 0.84 & 5.39e-11 \\
\multicolumn{2}{c}{} & \multicolumn{2}{c}{$2^{-7}$} & \multicolumn{2}{c}{4129024} & 3 & 5.39 & 1.13e-10 & 7 & 7.51 & 7.28e-10 \\
\multicolumn{2}{c}{} & \multicolumn{2}{c}{$2^{-8}$} & \multicolumn{2}{c}{33292800} & 3 & 59.97 & 1.14e-10 & 9 & 107.00 & 3.90e-11 \\ \midrule
\multicolumn{2}{c}{\multirow{4}{*}{$10^{-6}$}} & \multicolumn{2}{c}{$2^{-5}$} & \multicolumn{2}{c}{61504} & 3 & 0.17 & 2.53e-9 & 7 & 0.24 & 2.79e-9 \\
\multicolumn{2}{c}{} & \multicolumn{2}{c}{$2^{-6}$} & \multicolumn{2}{c}{508032} & 3 & 0.81 & 1.26e-9 & 10 & 1.65 & 5.65e-10 \\
\multicolumn{2}{c}{} & \multicolumn{2}{c}{$2^{-7}$} & \multicolumn{2}{c}{4129024} & 3 & 6.71 & 1.14e-9 & 10 & 11.98 & 3.22e-10 \\
\multicolumn{2}{c}{} & \multicolumn{2}{c}{$2^{-8}$} & \multicolumn{2}{c}{33292800} & 3 & 74.15 & 1.14e-9 & 10 & 126.17 & 5.38e-10 \\ \midrule
\multicolumn{2}{c}{\multirow{4}{*}{$10^{-4}$}} & \multicolumn{2}{c}{$2^{-5}$} & \multicolumn{2}{c}{61504} & 5 & 0.25 & 1.53e-7 & 14 & 0.42 & 1.53e-7 \\
\multicolumn{2}{c}{} & \multicolumn{2}{c}{$2^{-6}$} & \multicolumn{2}{c}{508032} & 5 & 1.24 & 3.40e-8 & 15 & 2.57 & 3.40e-8 \\
\multicolumn{2}{c}{} & \multicolumn{2}{c}{$2^{-7}$} & \multicolumn{2}{c}{4129024} & 5 & 10.89 & 8.51e-9 & 18 & 23.00 & 8.51e-9 \\
\multicolumn{2}{c}{} & \multicolumn{2}{c}{$2^{-8}$} & \multicolumn{2}{c}{33292800} & 5 & 100.60 & 2.13e-9 & 23 & 313.30 & 2.13e-9 \\ \midrule
\multicolumn{2}{c}{\multirow{4}{*}{$10^{-2}$}} & \multicolumn{2}{c}{$2^{-5}$} & \multicolumn{2}{c}{61504} & 5 & 0.25 & 1.16e-5 & 20 & 0.57 & 1.16e-5 \\
\multicolumn{2}{c}{} & \multicolumn{2}{c}{$2^{-6}$} & \multicolumn{2}{c}{508032} & 5 & 1.54 & 2.90e-6 & 24 & 4.28 & 2.90e-6 \\
\multicolumn{2}{c}{} & \multicolumn{2}{c}{$2^{-7}$} & \multicolumn{2}{c}{4129024} & 5 & 12.99 & 7.25e-7 & 30 & 39.82 & 7.25e-7 \\
\multicolumn{2}{c}{} & \multicolumn{2}{c}{$2^{-8}$} & \multicolumn{2}{c}{33292800} & 5 & 121.50 & 1.81e-7 & 52 & 739.91 & 1.81e-7 \\ \bottomrule
\end{tabular}
\end{table}

\begin{table}[!tbp]
\caption{Results of GMRES-$\mathcal{P}_{S}$ for Example \ref{ex:example4} with $\zeta=\pi$ and $\theta=\frac{1}{2}$ (Crank-Nicolson)}
\label{tab:table4}
\centering
\begin{tabular}{@{}clclclccc@{}}
\toprule
\multicolumn{2}{c}{\multirow{2}{*}{$\gamma$}} & \multicolumn{2}{c}{\multirow{2}{*}{$h$}} & \multicolumn{2}{c}{\multirow{2}{*}{DoF}} & \multicolumn{3}{c}{GMRES-$\mathcal{P}_{S}$} \\ \cmidrule(l){7-9} 
\multicolumn{2}{c}{} & \multicolumn{2}{c}{} & \multicolumn{2}{c}{} & Iter & CPU & $e_{h}$ \\ \midrule
\multicolumn{2}{c}{\multirow{4}{*}{$10^{-10}$}} & \multicolumn{2}{c}{$2^{-5}$} & \multicolumn{2}{c}{65536} & 3 & 0.23 & 1.51e-11 \\
\multicolumn{2}{c}{} & \multicolumn{2}{c}{$2^{-6}$} & \multicolumn{2}{c}{524288} & 3 & 1.02 & 1.39e-11 \\
\multicolumn{2}{c}{} & \multicolumn{2}{c}{$2^{-7}$} & \multicolumn{2}{c}{4194304} & 3 & 6.98 & 1.18e-11 \\
\multicolumn{2}{c}{} & \multicolumn{2}{c}{$2^{-8}$} & \multicolumn{2}{c}{33554432} & 3 & 81.18 & 4.99e-12 \\ \midrule
\multicolumn{2}{c}{\multirow{4}{*}{$10^{-8}$}} & \multicolumn{2}{c}{$2^{-5}$} & \multicolumn{2}{c}{65536} & 3 & 0.18 & 1.43e-9 \\
\multicolumn{2}{c}{} & \multicolumn{2}{c}{$2^{-6}$} & \multicolumn{2}{c}{524288} & 3 & 1.07 & 7.93e-10 \\
\multicolumn{2}{c}{} & \multicolumn{2}{c}{$2^{-7}$} & \multicolumn{2}{c}{4194304} & 3 & 8.37 & 2.06e-10 \\
\multicolumn{2}{c}{} & \multicolumn{2}{c}{$2^{-8}$} & \multicolumn{2}{c}{33554432} & 3 & 94.33 & 5.05e-11 \\ \midrule
\multicolumn{2}{c}{\multirow{4}{*}{$10^{-6}$}} & \multicolumn{2}{c}{$2^{-5}$} & \multicolumn{2}{c}{65536} & 3 & 0.22 & 3.69e-8 \\
\multicolumn{2}{c}{} & \multicolumn{2}{c}{$2^{-6}$} & \multicolumn{2}{c}{524288} & 3 & 1.26 & 8.56e-9 \\
\multicolumn{2}{c}{} & \multicolumn{2}{c}{$2^{-7}$} & \multicolumn{2}{c}{4194304} & 3 & 12.36 & 2.06e-9 \\
\multicolumn{2}{c}{} & \multicolumn{2}{c}{$2^{-8}$} & \multicolumn{2}{c}{33554432} & 3 & 152.13 & 5.05e-10 \\ \midrule
\multicolumn{2}{c}{\multirow{4}{*}{$10^{-4}$}} & \multicolumn{2}{c}{$2^{-5}$} & \multicolumn{2}{c}{65536} & 3 & 0.27 & 3.73e-7 \\
\multicolumn{2}{c}{} & \multicolumn{2}{c}{$2^{-6}$} & \multicolumn{2}{c}{524288} & 3 & 1.70 & 8.64e-8 \\
\multicolumn{2}{c}{} & \multicolumn{2}{c}{$2^{-7}$} & \multicolumn{2}{c}{4194304} & 3 & 18.68 & 2.08e-8 \\
\multicolumn{2}{c}{} & \multicolumn{2}{c}{$2^{-8}$} & \multicolumn{2}{c}{33554432} & 3 & 249.61 & 5.10e-9 \\ \midrule
\multicolumn{2}{c}{\multirow{4}{*}{$10^{-2}$}} & \multicolumn{2}{c}{$2^{-5}$} & \multicolumn{2}{c}{65536} & 3 & 0.35 & 4.10e-6 \\
\multicolumn{2}{c}{} & \multicolumn{2}{c}{$2^{-6}$} & \multicolumn{2}{c}{524288} & 3 & 2.28 & 9.51e-7 \\
\multicolumn{2}{c}{} & \multicolumn{2}{c}{$2^{-7}$} & \multicolumn{2}{c}{4194304} & 3 & 22.44 & 2.29e-7 \\
\multicolumn{2}{c}{} & \multicolumn{2}{c}{$2^{-8}$} & \multicolumn{2}{c}{33554432} & 3 & 282.79 & 5.61e-8 \\ \bottomrule
\end{tabular}
\end{table}

\end{document}